\documentclass[12 pt]{article}
\usepackage{graphicx} 
\usepackage{times}\usepackage{setspace}
\usepackage{graphicx}
\usepackage{float}
\usepackage[caption = false]{subfig}
\usepackage{hyperref}
\usepackage{natbib}
\usepackage{amssymb}
\usepackage{amsthm}
\newtheorem{dfn}{Definition}

\usepackage[english]{babel}

\newtheorem{theorem}{Theorem}
\newtheorem{lemma}[theorem]{Lemma}
\usepackage{enumerate}
\usepackage{nccmath}
\usepackage{setspace}
\usepackage{amsmath}
\usepackage{blindtext, multicol}
\usepackage{multicol}
\usepackage{mathrsfs}
\usepackage[autostyle, english = american]{csquotes}
\MakeOuterQuote{"}
\newtheorem{Lemma}{LEMMA}[section]
\newtheorem{Proposition}[Lemma]{Proposition}
\def\bp{\begin{Proposition}}
\def\ep{\end{Proposition}}
\newtheorem{prop}{Proposition}
\newtheorem{cor}[theorem]{Corollary}
\allowdisplaybreaks
\makeatletter
\renewcommand*{\@fnsymbol}[1]{\@arabic{#1}}
\makeatother

	\newcommand{\n}[1]{\lVert#1\rVert} 
	\newcommand{\la}[1]{\langle#1\rangle} 
	\newcommand{\p}[1]{{#1}_{\rho}} 
	\providecommand{\keywords}[1]
	{
\textbf{Keywords: } #1
}

\title{\Large Convergence Analysis of Statistical Inverse Problems on Reproducing Kernel Banach Spaces  }
\author{\small Darrel K Joseph\thanks{The author was supported through the Junior Research Fellowship by the University Grants Commission (UGC),
	India-NTA Ref. no.: 221610060111.\\Indian Institute of Science Education and Research, Thiruvananthapuram, India. \texttt{darreljoseph23@iisertvm.ac.in}}
$~$ and $~$ M P Rajan\thanks{Indian Institute of Science Education and Research, Thiruvananthapuram, India. \texttt{rajanmp@iisertvm.ac.in}.
	The author acknowledges the support received from the SERB-CRG grant no. CRG/2023/004903.}}
	\date{}
	\begin{document}

\maketitle
\noindent\textbf{Abstract}
Statistical inverse problems have garnered significant attention
in recent years due to the growing importance of statistical
learning theory and functional analytic approaches in the fields
of machine learning and artificial intelligence. In this paper, we
investigate the stable approximation of the element $u^{\dagger}$
that satisfies the equation $Au = g$, where $A$ is a linear
operator that maps a Banach space into an appropriate function space. The function $g$ is observed only
through independently and identically distributed data points that
are corrupted by noise and assumed to follow an unknown
distribution $\rho$. We employ the Tikhonov regularization scheme,
leveraging statistical learning techniques and the framework of
reproducing kernel Banach spaces to estimate the solution. We
establish convergence and derive the convergence rate of the
estimated solution with respect to the true solution as the number
of data points increases, with the rate expressed in probabilistic
terms. The theoretical findings are further supported by numerical
experiments that demonstrate the effectiveness of the proposed
approach.


\noindent \keywords{Reproducing kernel Banach space $\cdot$ Inverse problem $\cdot$ Statistical learning $\cdot$ Convex analysis $\cdot$ Regularization}

\section{Introduction}

In the classical inverse problem literature, we deal with equations of the form
\[A(x)=y,\]
where $A$ is a bounded linear operator between appropriate vector spaces. We are interested in approximating the solution $x$ for some given $y$.
If $A$ is ill-posed, that is if $A$ has no continuous inverse, then a small perturbation in the value of $y$ can cause a large error for the output $x$,
hence different regularization techniques are employed to find a stable approximation.

\noindent For deterministic cases, we often have to deal with $y^{\delta}$ instead of $y$ such that
\[\n{y-y^{\delta}}\leq\delta,\qquad \delta>0\]
Then, we solve for the equation modeled as follows
\[A(x)=y^{\delta}.\]
If  $x^{\dagger}$ is the true solution of $y$, i.e., $A(x^{\dagger})=y$, then we seek out a stable approximation $x_{\lambda}^{\delta}$ of $x^{\dagger}$ such that \[\n{x_{\lambda}^{\delta}-x^{\dagger}}\to0\qquad\text{as}\qquad \delta,\lambda\to0.\]
Here, $\lambda$ is the regularization parameter and $x_{\lambda}^{\delta}$ is the regularized solution corresponding to $y^{\delta}$.
This has been well studied during the late twentieth and early twenty-first century, and
the prominent work can be seen in
\cite{engl2000regularization,nair2009linear,schuster2012regularization} and in many works of the author (cf \cite{rajan1, rajan2, rajan3,rajan4,rajan5,rajan6}).\\

\noindent
Inverse problems in the statistical learning setting have
attracted growing interest in recent years. Unlike the
deterministic case, the error here is modeled as a random
variable following a certain probability distribution.
Accordingly, the convergence rate and the error between the true
and regularized solutions are expressed in terms of expectations
or probabilities. This formulation aligns closely with the
framework of statistical learning. The advantage of statistical
inverse problems over classical ones lies primarily in the ability
of statistical and Bayesian approaches to address ill-posedness
and uncertainty more effectively. A key strength of statistical
inverse problems is that they provide probabilistic solutions that
quantify uncertainty, incorporate prior information, and yield a
distribution of possible outcomes rather than a single
deterministic solution. In contrast, classical inverse problems
often face challenges such as non-uniqueness, sensitivity to
noise, and the absence of a formal mechanism for uncertainty
quantification. This probabilistic and principled framework for
integrating data and prior knowledge makes statistical inverse
problems particularly powerful in modern, data-driven applications
where noise and model uncertainty are significant concerns.
Several recent works have addressed these challenges (see
\cite{BGM, BAUER200752, Liu2017, Learningrate}).\\

In statistical inverse problem, we consider a linear operator $A$ from a normed space $B$ to a linear space $V$ consisting of real valued functions on $X$. For $M>0$, given a set of $m$ data points of the form $(x_i,y_i)\in (X\times[-M,M])$ that follows some probability measure $\rho(x,y)$, we aim to approximate a minimizer $u_{\rho}\in B_1$ of the risk functional
$$\int_{X\times[-M,M]}|y-(Au)(x)|^pd\rho,~~p>1$$ 
over all $u\in B$. In \cite{BGM}, such an inverse learning problem is studied for $p=2$ when $B$ is a Hilbert space, for which an optimal rate is obtained. Whereas in \cite{Liu2017, Learningrate}, the study is carried out for $p=2$ in a Banach space setting for the special case $A=I$, which is also called the direct problem. The convergence and the convergence rate is given in terms of $L^2$ norm. We extend this framework to the inverse problem setting in Banach spaces by introducing a more general bounded operator $A$, thereby providing a unified approach to both the direct and inverse problems. We consider an empirical $p$-th power loss instead of squared loss. The convergence rates are then characterized in terms of the norm of the underlying space. \\

In statistical learning, the notion of convergence is based on the norm difference between the reconstructed solution $\hat{u}$ and $u_{\rho}$ as $m\to\infty$. A regularization technique analogous to Tikhonov regularization
\cite{Tikhonov:1963} is employed. The method relies on the concept of Reproducing Kernel Banach Space (RKBS), first introduced in \cite{5179093}. From a mathematical point of view, the learning problem involves minimizing a functional with an $L^p$- type loss $(p>1)$ along with a norm-based regularization in a Banach space. This setting does not rely on an inner product structure, so standard Hilbert space tools are not directly applicable. The RKBS framework provides a suitable setting where point evaluations are still continuous, ensuring that the risk functional is well-defined and the optimization problem is meaningful. The RKBS structure also allows us to write the function values in terms of the elements of the dual space, using a Kernel function. The main advantage is that the subdifferentials can then be expressed using the given Kernel function which eventually aids us to approximate the solution. A more generalized version of RKBS is given in \cite{Xu2023SparseBanach}, where the reflexivity of the underlying space is not required. Such a general framework allows us to consider  non-reflexice spaces as well. However, in our paper we consider a reflexive space and hence the original definition of RKBS is adopted. On the other hand, \cite{JMLR:v22:20-751} and \cite{MR4749129} have used the more recent definition of RKBS to establish representer theorem, and solutions for various inverse problem independent of the reflexivity of the space. The key difference between our paper and theirs is that they have considered a functional analytic approach, whereas we use statistical approach to determine a high probability convergence.\\

\noindent The paper is organized as follows: Section \ref{sec2}
introduces the problem setting, basic definitions and necessary
assumptions. Section \ref{sec3}
presents the main theoretical results and corollary. In Section \ref{proofs} the proofs establishing
convergence and convergence rates are given. Finally, Section \ref{sec4}
provides a numerical example to demonstrate and validate the
theoretical findings.

\section{Problem setting}
\label{sec2}
\noindent In this paper, we consider an injective bounded linear operator
$A$ from a real Banach space $(B_1,\parallel\cdot\parallel_{B_1})$ to
a vector space $V$ of real-valued functions over some compact metric space $X$, i.e
\[A:B_1\to V\subseteq \mathcal{F}(X,\mathbb{R}).\]
Then our problem will be modeled as
\[\qquad Au=g,\qquad u\in B_1~\text{and}~ g\in V,\] where
$g$ is observed only through independently and identically
distributed noisy data points $z=\{z_i\}_{i=1}^m=\{(x_i,y_i)\}_{i=1}^m\in Z^m=(X,[-M,M])^m, M>0$
that are assumed to follow an unknown Borel probability distribution
$\rho(x,y).$
The distribution will be considered to be of the form $\rho(x,y)=\rho(y|x)\rho_X(x)$,
where $\rho(y|x)$ is the conditional distribution of $y$ for some given $x$ and $\rho_X(x)$ is the marginal distribution.
\vspace{0.2cm}

\noindent For some $p>1$, we assume that there exists a minimizer $u_{\rho}\in B_1$ corresponding to the measure $\rho(x,y)$ on $Z=(X\times [-M,M])$ such that
\begin{align}\label{true}
	u_{\rho}:&=\arg \min _{u\in B_1} \int_{Z}|y-(Au)(x)|^{p}~ d \rho\\
	&=\arg \min _{u\in B_1}\mathcal{E_{\rho}}(u)\notag,
\end{align}where
\begin{equation}\label{err}
	\mathcal{E_{\rho}}(u):= \int_{Z}|y-(Au)(x)|^{p}~ d \rho.
\end{equation}
We will show that in our settings the image of $A$ will be contained in $C(X)$ and therefore \eqref{true} is integrable for all $p$.
\vspace{0.3cm}

\noindent The regularized solution corresponding to the data points $z$ is given by $u_z^{\lambda}\in B_1$, such that

\begin{equation}\label{reg}
	u_z^{\lambda}=\text{arg}\underset{u\in B_1}{\text{min}}\left(\frac{1}{m}\sum_{i=1}^m|y_i-(Au)(x_i)|^p+\frac{\lambda}{q} \n u ^q_{B_1}\right)
\end{equation}
where $q>1$ and $\lambda$ is the regularization parameter. Note that we only used the vector space structure of the co-domain.

\noindent We go on to further define the $\rho$ induced regularized solution  corresponding to $\lambda$ and \eqref{true} by,
\begin{equation}\label{regtrue}
	u^{\lambda}_{\rho}:= \arg\min_{u\in B_1}\left(\mathcal{E}_{\rho}(u)+\frac{\lambda}{q}\n u_{B_1}^q\right).
\end{equation}
The minimizers \eqref{reg} and \eqref{regtrue} exist uniquely from the fact that the expression in the parentheses is a strictly convex, continuous, and coercive function (\cite{MR463994}, proposition
$1.2$).\\
The objective is to establish the convergence of $u_z^{\lambda}$
to $u_{\rho}$ and derive the convergence rate.
\subsection{Definitions and basic results}
In this subsection, we give an overview of some of the basic  definitions that we used for analysis purposes from convex analysis and related literature.
\begin{dfn} Subdifferential:
	Let $F:B\rightarrow\mathbb{R}$ be a convex function defined on a
	Banach space $B$. The subdifferential $\partial F$ of $F$ at the
	point $u\in B$ is characterized as the set
	\cite{rockafellar1970convex}
	\begin{align}\label{subd}
		(\partial F)(u)&=\{ u^* \in {B}^{*}: F\left(v\right)-F(u)
		\left.\geq\left\langle v-u,u^* \right\rangle_{B\times B^*}, \forall v \in {B}\right\}\\
		& \neq \phi .\notag
	\end{align}
\end{dfn}
%
\noindent Here, $\langle \cdot,\cdot\rangle_{B\times B^*}$ is the
duality product. If $u^* \in$ $(\partial F)(u)$, then we call
$u^*$ a subgradient of $F$ at $u$.
We note that $0\in(\partial F)(u_0)$ if and only if $u_0$ is a point of minimum \cite{Clarke1998}.
Further, whenever the function $F$ has the form $F(u)=\frac{1}{q}\|u\|_{{B}}^{q},q>1$, we will denote the subdifferentials of $F$ at $u\in B$ as $J_{q}(u)$.
\begin{dfn} $q$-Uniform Convex Space:
	The Banach space ${B}$ is called a $q$-uniform convex space if there exists a constant $c_{q}>0$ for which
	\begin{equation}\label{eq6}
		\begin{array}{r}
			\frac{1}{q}\left(\|u\|_{{B}}^{q}-\|v\|_{{B}}^{q}\right) \geq\left\langle u-v,j_q(v)\right\rangle_{B^*\times B}+c_{q}\|u-v\|_{{B}}^{q}
		\end{array}
	\end{equation}
	for all $u$, $v \in {B}$ and all $j_{q}(v) \in J_{q}(v)$ \cite{Xu1991Inequalities}.
\end{dfn}
\noindent    The parallelogram law ensures that Hilbert spaces are
$2$-uniform convex. The sequence space $l^p$, the equivalence
classes of measurable function space $L^p$, the Sobolev space
$W^{k,p}$, etc, are $p$-uniform convex whenever $p\geq2$ and $2$
-uniform convex when $1<p<2$ \cite{Hanner1956,XuRoach1991}.
\vspace{0.2cm}

\begin{dfn} Reproducing Kernel Banach Space\cite{5179093}: \label{rkbs}
	A reflexive Banach space $B$ of functions on $X$ is said to be a reproducing kernel Banach space (RKBS)
	if there is another Banach space $B^{\#}$ of functions on $X$ that is isometrically isomorphic to $B^*$ (the continuous dual of $B$)
	such that the point evaluations are continuous on $B$ and $B^{\#}$.
\end{dfn}

\begin{dfn}\label{def1}
	Let $B_K$ be a Banach space defined by taking the elements of the range space, that
	is,
	\begin{equation}
		\label{B_k}
		B_K :=\{g\in V~|~\exists u\in B_1 ~\text{such that}~ Au=g  \}
	\end{equation}
	with the associated norm
	\begin{equation*}
		\n g _{B_K}:={\n u _{B_1}~~\text{where}~ u\in B_1~\text{such that}~
			Au=g}.
	\end{equation*}
	
\end{dfn}
\noindent We will show later that $B_K$ is an RKBS.

\subsection{Assumptions}
We now make the following assumptions:
\begin{enumerate}
	\item[(A1)] There exists a Banach space $B^{\#}$ of functions from $X$ to $\mathbb{R}$ which is isometrically isomorphic to $B_1^*$ such that the pointwise evaluation linear maps are continuous on $B^{\#}$.
	\item[(A2)] The evaluation functionals with respect to the sample point $x\in X$
	\[{S_x:B_1\rightarrow\mathbb{R}}\]
	\begin{equation}\label{sam}
		S_x(u)=(Au)(x)
	\end{equation}
	are uniformly continuous, that is, there exists a constant $0<k<\infty$ such that\\
	\begin{equation}\label{bound}
		{|S_x(u)|\leq k\n u_{B_1}}\quad\forall x\in X.
	\end{equation}
\end{enumerate}
\begin{enumerate}
	\item[(A3)] We assume that the approximation error satisfies
	\begin{equation}\label{eq9}
		\|u^{\lambda}_{\rho}-u_{\rho}\|_{B_1}\leq c_{\beta}\lambda^{\beta}
	\end{equation}
	where $\beta$ and $c_{\beta}$ are positive constants.
\end{enumerate}
We note that the bounds given by the third assumption are common in inverse problem literature (cf.  Proposition $5.8$ in
\cite{BGM}, and \cite{Hofmann2007}). Such a bound can also be derived from appropriate source conditions as well (cf. Chapter $4.4$ \cite{nair2009linear}). Now we claim that $B_K$ is an
RKBS under assumption $1$ and $2$.

\begin{prop} Let $B_1$ be a $q$-uniform convex Banach space. Then the Banach space $B_K$ defined in (\ref{B_k}) is an
	RKBS.
\end{prop}
\begin{proof}  We assert that $B_K$ is well defined as $A$ is injective. Further, $A$
	is an isometric isomorphism between $B_1$ and $B_K$ as it is a
	bijective linear map that preserves norm. Hence, it follows that
	$B_1$ and $B_K$ are isometrically isomorphic. Similarly, we can
	see that $B_1^*$ and $B_K^*$ are isometrically isomorphic. Let
	$g^*\in B_K^*$, then the adjoint map defined as
	\begin{align}
		A^*&:B_K^*\to B_1^*\\\notag
		&A^*(g^*)=g^*(A)
	\end{align}
	is a bijective isometry. Hence, by assumption $(A1)$ , $B_K^*$ and $B^{\#}$ are isometrically isomorphic.
	
	We note that  $$|g(x)|=|(Au)(x)|=|S_x(u)|\leq k\n u_{B_1}=k\n g_{B_K}.$$
	This implies that the point evaluations are continuous on $B_K$, whose continuous dual is isometrically isomorphic to $B^{\#}$. Further, since $B_1$ is $q$-uniform convex (hence uniformly convex), it is a reflexive Banach space which makes $B_K$ to be reflexive.
	Then by definition \ref{rkbs}, $B_K$ is an RKBS on $X$. \vspace{0.3cm}
\end{proof}
\noindent {\bf Remark:}    From \cite{5179093}, one can find a
reproducing kernel $K:X\times X\rightarrow \mathbb{R}$ that
satisfies the following conditions
\begin{align}
	g(x)&=\langle g, K(\cdot,x)\rangle_{B_K\times B_K^*}\quad\forall g\in B_K~\text{and}~ \forall ~x\in X\\
	g^*(x)&=\la{K(x,\cdot),g^*}_{B_K\times B_K^*}\quad \forall g^*\in B_K^*~\text{and}~ \forall ~x\in X\notag
\end{align}
By the isometry of $B_1$ and $B_K$, the above equations give
\begin{align}\label{rep}
	g(x)&=\langle u, H(\cdot,x)\rangle_{B_1\times B_1^*}\quad u\in B_1,H(\cdot,x)\in B_1^*\\
	g^*(x)&=\langle H(x,\cdot),u^*\rangle_{B_1\times B_1^*}\quad H(x,\cdot)\in B_1,u^*\in
	B_1^*\notag,
\end{align}
where we had $Au=g$ and $H(\cdot,x)=A^*(K(\cdot,x))$. A similar argument tells us $A(H(x,\cdot))=K(x,\cdot)$ and $A^*g^*=u^*$.

\vspace{0.2cm}
\noindent   It can be easily verified from the previous statement that
\begin{equation}\label{kxy}
	K(x,y)=\langle K(x,\cdot),K(\cdot,y)\rangle_{B_K\times B_K^*}=\langle H(x,\cdot),H(\cdot,y)\rangle_{B_1\times B_1^*}
\end{equation}

\noindent    From the properties mentioned above, the reproducing
property for any function $g\in B_K$ can be stated as
\begin{equation}\label{repty}
	g(x)=\langle u,H(\cdot,x)\rangle_{B_1\times B_1^*}\quad\forall x\in X
\end{equation}
where $u$ being the only element such that $Au=g$.\\

\noindent\textbf{Example:} Before moving into further definitions and concepts, an example of an RKBS $B_K$ and an operator $A$ is provided. (See \cite{5179093}, equation $(19)$).\\

Let $X=\mathbb{R},I=[-\frac{1}{2},\frac{1}{2}]$ and $B_1=L^p(I)$, $1<p<2$. Let the operator $$A:B_1\to C
(\mathbb{R})$$ be given as
\begin{equation*}
	(Af)(x)=\int_If(t)e^{i2\pi xt}dt
\end{equation*} 
Then the set $B_K$ defined as
\begin{equation*}
	B_K=\{g\in C(\mathbb{R})~|~g=Af~\text{for some}~f\in L^p(I)\}
\end{equation*}
with associated norm $\|g\|_{B_K}=\|f\|_{L^p}$ is an RKBS on $\mathbb{R}$.
The dual space will be of the form
\begin{equation*}
	B_K^*=\{\phi\in C(\mathbb{R})~|~\phi(x)=\int_Ie^{-i2\pi xt}\psi(t)dt~\text{for some}~\psi\in L^q(I)\}
\end{equation*}
with norm $\|\phi\|_{B_K^*}=\|\psi\|_{L^q}$ and the duality product
\begin{equation*}
	\langle g,\phi\rangle_{B_K\times B_k^*}=\langle f,\psi\rangle_{L^p\times L^q}
\end{equation*}
The kernel function is then given by
\begin{equation*}
	K(x,y)=\frac{\sin\pi(x-y)}{\pi(x-y)}
\end{equation*}
From \cite{5179093}, it can be verified that $A$ is injective (as $L^p(I)\subset L^1(I)$) and $$g(x)=\langle g,K(\cdot,x)\rangle_{B_K\times B_k^*}=\langle f,H(\cdot,x)\rangle_{B_1\times B_1^*},$$ where $H(\cdot,x)$ is defined as $A^*K(\cdot,x)$. 
Further, we have
$$|S_x^*(f)|=|(Af)(x)|\leq \int_I|f(t)e^{i2\pi xt}|dt\leq\int_I|f(t)|\cdot1dt\leq\|f\|_{L^p}$$
where the last inequality follows from Holder's inequality. Thus the point evaluation functionals are uniformly continuous. In addition $B_1^*\cong L^q(I)$ is isometric to $B_K^*$, a function space on $\mathbb{R}$. Hence, $B$ is an RKBS on $X$. \\

\noindent We now continue with the remaining definitions.
\begin{dfn} Covering Number\cite{dudley1967sizes}:
	The covering number $\mathcal{N}(X,r)$ of a metric space $X$ is the minimum number of open balls of radius $r$ required to cover $X$.
\end{dfn}

\begin{dfn} Logarithmic Complexity:
	A compact subset $E$ of a Banach space $\left({B},\|\cdot\|_{{B}}\right)$ is said to have logarithmic complexity exponent $s \geq 0$ if there exists a $c_{s}>0$ for which ${B}_E^R=\left\{f \in E:\|f\|_{{B}} \leq R\right\}$ satisfies
\end{dfn}
\begin{equation}\label{cov}
	\mathcal{N}\left({B}_E^R, r\right) \leq e^{c_{s}\left(\frac{R}{r}\right)^{s}}, \quad \forall r>0
\end{equation}
If we choose $r=\frac{1}{n},n\in\mathbb{N}$, then we have
\begin{equation*}
	\mathcal{N}\left({B}_E^R,\frac{1}{n} \right) \leq
	e^{c_{s}\left(Rn\right)^{s}}.
\end{equation*}

\noindent {\bf Notation:} Let
\begin{equation}
	\label{eqnW}
	W_{\rho}^{\lambda}(x,y)=|y-(Au_{\rho}^{\lambda})(x)|^{p-1}\text{sgn}[y-(Au_{\rho}^{\lambda})(x)].
\end{equation}
\noindent Then, for $h_u(x,y):=W_{\rho}^{\lambda}(x,y)(Au)(x)$, we define
the set $\mathcal{G}_{\rho}^{\lambda}$ as

\begin{equation}\label{G}
	\mathcal{G}_{\rho}^{\lambda}=\left\{h_u:X\times[-M,M]\to\mathbb{R}~|~ u \in {B_1},\|u\|_{{B_1}} \leq
	1\right\}.
\end{equation}
We will show that this is a compact subset of $C(X\times[-M,M])$ and hence has a finite covering number (which will be used later).
\noindent Further, we say that $H(\cdot,x)\in B_1^*$ is a uniformly continuous function with respect to $\|\cdot\|_{B_1^*}$ if $\forall\epsilon>0$, $\exists\delta>0$ depending on $\epsilon$ such that
\[\|H(\cdot,x)-H(\cdot,y)\|_{B_1^*}<\epsilon\]
whenever $d(x,y)<\delta$.\vspace{0.3cm}


\section{Main Results}
\label{sec3}
In this section, we present the major results associated with
the convergence and convergence rate of the proposed scheme. We will state the theorems and the associated corollary.

\begin{theorem}\label{thm1}
	Let $(B_1,\|\cdot\|_{B_1})$ be a $q$-uniform convex Banach space and $V$ be a real vector space of functions
	that maps a compact metric space $X$ to $\mathbb{R}$. Consider $A:B_1\to V$ to be an injective linear operator
	such that the assumptions $(A1)$, $(A2)$, and $(A3)$ are true. If $H(\cdot,x)$ is a uniformly continuous function with respect to $\|\cdot\|_{B_1^*}$ and $\mathcal{G}_{\rho}^{\lambda}$ has logarithmic complexity $s$, then for any $\delta\in(0,1)$, there holds with confidence $1-\delta$ that
	\begin{equation}\label{eq22}
		\left\|u_{z}^{\lambda}-u_{\rho}\right\|_{B_1} \leq  \sqrt[q-1]{\frac{4p\Omega}{\lambda c_q}}~\max\left[\left(\frac{ \log (2 / \delta)}{m}\right)^{\frac{1}{2(q-1)}},\left(\frac{c_{s}}{m}
		\right)^{\frac{1}{(2+s)(q-1)}}\right]+c_{\beta}{\lambda}^{\beta},
	\end{equation}
	where
	\begin{equation}\label{omega}
		\Omega=\Omega(\lambda):=(M+k\|u_{\rho}\|_{B_1}+kc_{\beta}\lambda^{\beta})^{p-1}k.
	\end{equation}
\end{theorem}

\noindent    In the next theorem, we show that as $m$ increases,
the expression in the RHS of Theorem \ref{thm1} can be more
accurately determined.
\vspace{0.3cm}
\begin{theorem}\label{thm2}
	Let $\delta\in(0,1)$ and suppose that the assumptions of Theorem \ref{thm1} hold true. Then, if $m$ is large enough, there holds with confidence $1-\delta$ that
	\begin{equation}\label{eqthm2}
		\left\|u_z^{\lambda}-u_{\rho}\right\|_{B_1} \leq  \sqrt[q-1]{\frac{4p\Omega}{\lambda
				c_q}}\times\sqrt[(2+s)(q-1)]{\frac{c_{s}}{m}}+c_{\beta}{\lambda}^{\beta}.
	\end{equation}
\end{theorem} 

\noindent Finally we present a corollary of the above theorem.
\begin{cor}\label{thm3}
	Let $\delta\in(0,1)$ and suppose that the assumptions of Theorem \ref{thm1} and \ref{thm2} hold true. Then, for
	\begin{equation}\label{lambda}
		\lambda={m^{-\frac{1}{(2+s)(1+\beta(q-1))}}}
	\end{equation}
	with a confidence of $1-\delta$, we have the following convergence rate
	\begin{equation}\label{conv}
		\left\|u_z^{\lambda}-u_{\rho}\right\|_{B_1}=O(\lambda^{\beta}).
	\end{equation}
\end{cor}

\section{Proofs}
\label{proofs}
Now we present the proofs for the above stated theorems and corollary. We obtain a high probability upper bound for $\|u_z^{\lambda}-u_{\rho}\|_{B_1}$ in terms of the number of data points $m$. This is done by splitting the total error into sample error $\|u_z^{\lambda}-u_{\rho}^{\lambda}\|_{B_1}$ and the approximation error  $\|u_{\rho}^{\lambda}-u_{\rho}\|_{B_1}$, and subsequently establish an upper bound for the former.

The proofs are presented as a sequence of lemmas. In lemma \ref{lemma4} and \ref{lemma5}, we establish an explicit form for $\partial\mathcal{E}_{\rho}$ and $j_q(u_{\rho}^{\lambda})$. Lemma \ref{lemma7} shows an upper bound for the difference between $u_{\rho}^{\lambda}$ and $u_{\gamma}^{\lambda}$, where the latter is the regularized solution w.r.t the distribution $\gamma$. By the end of lemma \ref{lemma9}, we consider $\gamma$ to be the empirical distribution to obtain a bound for $\|u_z^{\lambda}-u_{\rho}^{\lambda}\|_{B_1}$ in terms of the covering number. Finally, we combine these lemmas to prove the stated theorems.

Now we start with the following lemma.

\begin{lemma}\label{lemma4}
	Let $(B_1,\|\cdot\|_{B_1})$ be a $q$-uniform convex Banach space and $V$ be a real vector space of functions
	that maps a compact metric space $X$ to $\mathbb{R}$. Consider $A:B_1\to V$ to be an injective linear operator
	such that the assumptions $(A1)$, $(A2)$ and $(A3)$ are true.
	Let $\mathcal{E}_{\rho}$ be as defined in \eqref{true}, then 
	for $u_{0} \in {B_1} $
	\begin{equation}
		(\partial \mathcal{E}_{\rho})\left(u_{0}\right)=\left\{-p \int_{Z}\Big(\left|y-(Au_0)(x)\right|^{p-1}\text{sgn}[y-(Au_0)(x)]\Big)H(\cdot,x) d
		\rho\right\}.
	\end{equation}
\end{lemma}
\begin{proof}
	Let us take the function $F:B_1\to\mathbb{R}$ identified as
	\begin{equation}
		F(u)=\left|y-(Au)(x)\right|^p=\left|y-S_x(u)\right|^p
	\end{equation}
	
	\noindent From \cite{Showalter1997}, Proposition $7.8$, it can be seen that at $u_0$
	\begin{equation*}
		(\partial F)(u_0)=\left\{-S_x^*~(\partial
		I){[y-(Au_0)(x)]}~\right\}.
	\end{equation*}
	
	\noindent   Here $I(r)=|r|^p,~I:\mathbb{R}\to\mathbb{R}$ is a real-valued differentiable function acting at the point $y-(Au_0)(x)$.
	From \cite{rockafellar1970convex} Theorem $25.1$, we use the property that the subdifferential of a differentiable function is the gradient to show that
	
	\begin{equation}\label{eq30}
		(\partial F)(u_0)=-S_x^*\Big(p
		\left|y-(Au_0)(x)\right|^{p-1}\text{sgn}[y-(Au_0)(x)]\Big).
	\end{equation}
	
	\noindent By using the fact that for any $r\in\mathbb{R}$ and $z\in B_1,$
	\begin{align*}
		\langle z,S_x^*r\rangle_{B_1\times B_1^*}=r(S_xz)=r\langle z,H(\cdot,x)\rangle_{B_1\times B_1^*}=\langle z,rH(\cdot,x)\rangle_{B_1\times
			B_1^*}.
	\end{align*}
	
	\noindent   We have $S_x^*r=rH(\cdot,x)$. Then \eqref{eq30} attains the form
	
	\begin{equation}
		(\partial F)(u_0)=-\Big(p \left|y-(Au_0)(x)\right|^{p-1}\text{sgn}[y-(Au_0)(x)]\Big)H(\cdot,x)
	\end{equation}
	
	\noindent   Now, by the definition of subgradient,
	\begin{equation*}
		F(v)-F(u_0)\geq\langle v-u_0,(\partial F)(u_0)\rangle_{B_1\times B_1^*}\quad\forall v\in
		B_1.
	\end{equation*}
	
	\noindent   By the properties of the integral, we have the following inequality
	\begin{align*}
		\int_Z F(v)-F(u_0)~d\rho&\geq\int_Z\langle v-u_0,(\partial F)(u_0)\rangle_{B_1\times B_1^*} d\rho\\
		&=\int_Z(\partial F)(u_0)(v-u_0)d\rho .
	\end{align*}
	
	\noindent   In view of the right hand side of the above relation, being an integral operator from $B_1$ to $\mathbb{R}$,
	\begin{align*}
		\int_Z F(v) ~d\rho-\int_Z F(u_0) ~d\rho \geq\Big\langle v-u_0,\int_Z(\partial F)(u_0)~d\rho\Big\rangle_{B_1\times
			B_1^*}.
	\end{align*}
	
	\noindent       From the definition of $F$, the following inequality is
	immediate.
	\begin{align*}
		\int_Z \Big|y-(Av)(x)\Big|^p d\rho-\int_Z \left|y-(Au_0)(x)\right|^p d\rho \geq\Big\langle v-u_0,\int_Z(\partial F)(u_0)~d\rho\Big\rangle_{B_1\times
			B_1^*}.
	\end{align*}
	
	\noindent   That is, $\int_Z(\partial F)(u_0)d\rho$ is the subgradient of $\mathcal{E}$ at $u_0$.
	Therefore, by substituting the value of $(\partial F)(u_0)$, we can see that
	\begin{equation*}
		(\partial \mathcal{E}_{\rho})\left(u_{0}\right)=\left\{-p \int_{Z}\Big( \left|y-(Au_0)(x)\right|^{p-1}\text{sgn}[y-(Au_0)(x)]\Big)H(\cdot,x) d
		\rho\right\}.
	\end{equation*}
	This proves the lemma.
\end{proof}
\vspace{0.1cm}

\begin{lemma}\label{lemma5}
	Suppose that $J_{q}(u)$ is as in \eqref{eq6} and $u_{\rho}^{\lambda}$ be defined as in \eqref{regtrue}.
	Then, there exists a $j_{q}\left(u_{\rho}^{\lambda}\right) \in J_{q}\left(u_{\rho}^{\lambda}\right)$ which satisfies the equation
	\begin{equation}\label{eq33}
		\lambda j_{q}\left(u_{\rho}^{\lambda}\right)=p \int_{Z}\Big( \left|y-(Au_{\rho}^{\lambda})(x)\right|^{p-1}\text{sgn}[y-(Au_{\rho}^{\lambda})(x)]\Big)H(\cdot,x) d
		\rho .
	\end{equation}
\end{lemma}
\begin{proof}
	Since ${u_{\rho}^{\lambda}}$ is the solution of \eqref{regtrue}, we have
	\begin{equation}\label{eq34}
		\left.0 \in
		\partial\left(\mathcal{E}_{\rho}(u)+\frac{\lambda}{q}\|u\|_{B_1}^{q}\right)\right|_{u=u_{\rho}^{\lambda}}.
	\end{equation}
	As both $\mathcal{E}_{\rho}$ and $\|\cdot\|_{B_1}^{q}$ are convex functions, and $\|\cdot\|_{B_1}^q$ is continuous, by Theorem $47B$(\cite{zeidler1985nonlinear}),
	\begin{align*}
		\left.\partial\left(\mathcal{E}_{\rho}(u)+\frac{\lambda}{q}\|u\|_{{B_1}}^{q}\right)\right|_{u=u_{\rho}^{\lambda}}
		=\left.\partial\left(\mathcal{E}_{\rho}(u)\right)\right|_{u=u_{\rho}^{\lambda}}+\left.\partial\left(\frac{\lambda}{q}\|u\|_{{B_1}}^{q}\right)\right|_{u=u_{\rho}^{\lambda}}
		.
	\end{align*}

	\noindent Hence, by Lemma \ref{lemma4} and \eqref{eq34},
	\begin{equation*}
		0 \in\left\{-p \int_{Z}\Big( \left|y-(Au_{\rho}^{\lambda})(x)\right|^{p-1}\text{sgn}[y-(Au_{\rho}^{\lambda})(x)]\Big)H(\cdot,x) d \rho\right\}+\lambda
		J_{q}\left(u_{\rho}^{\lambda}\right) .
	\end{equation*}
	
	\noindent Therefore, there exist some $j_{q}\left(u_{\rho}^{\lambda}\right) \in J_{q}\left(u_{\rho}^{\lambda}\right)$ such that
	\begin{equation*}
		0 =-p \int_{Z}\Big( \left|y-(Au_{\rho}^{\lambda})(x)\right|^{p-1}\text{sgn}[y-(Au_{\rho}^{\lambda})(x)]\Big)H(\cdot,x) d \rho+\lambda
		j_{q}\left(u_{\rho}^{\lambda}\right).
	\end{equation*}
	This shows \eqref{eq33} is true and hence the proof.
\end{proof}
\vspace{0.1cm}

\begin{lemma}\label{lemma7}
	Let $B_1$ be a $q$-uniform convex Banach space and $H(\cdot, x),~ \mathcal{E}_{\rho}(u)$ be defined respectively as in \eqref{Hxeqn} and \eqref{err}.
	If $u_{\rho}^{\lambda}$ and $u_{\gamma}^{\lambda}$ are the minimizers given by \eqref{regtrue} for the distributions $\rho$ and $\gamma$, respectively, then,
	\begin{equation}
		\label{eqnlemma51}
		\begin{aligned}
			\n{u_{\rho}^{\lambda}-u_{\gamma}^{\lambda}}_{B_1}\leq \Big(\frac{p}{\lambda c_q}\times \Big\|&\int_{Z}W_{\rho}^{\lambda}(x,y) H(\cdot, x) d \rho \\
			& -\int_{Z}W_{\rho}^{\lambda}(x,y) H(\cdot, x) d \gamma \Big\|_{B_1^*}\Big)^{1
				/(q-1)},
		\end{aligned}
	\end{equation}
	where $c_q$ is as in \eqref{eq6} and $W_{\rho}^{\lambda}(x,y)$ as in
	\eqref{eqnW}.
	\end{lemma}
	\begin{proof}
From the definition of $(\partial \mathcal{E}_{\gamma})\left(u_{\rho}^{\lambda}\right),$ we get
\begin{equation*}
	\mathcal{E}_{\gamma}\left(u_{\gamma}^{\lambda}\right)-\mathcal{E}_{\gamma}\left(u_{\rho}^{\lambda}\right)
	\geq\left\langle u_{\gamma}^{\lambda}-u_{\rho}^{\lambda},(\partial \mathcal{E}_{\gamma})\left(u_{\rho}^{\lambda}\right)\right\rangle_{B_1\times
		B_1^*}.
\end{equation*}
From Lemma \ref{lemma4}, we have,
\begin{align}\label{eq39}
	&\mathcal{E}_{\gamma}\left(u_{\gamma}^{\lambda}\right)-\mathcal{E}_{\gamma}\left(u_{\rho}^{\lambda}\right)
	\geq\notag\\
	&\Big\langle u_{\gamma}^{\lambda}-u_{\rho}^{\lambda},-p \int_{Z}\left|y-(Au_{\rho}^{\lambda})(x)\right|^{p-1}\text{sgn}[y-(Au_{\rho}^{\lambda})(x)] H(\cdot, x) d \gamma\Big\rangle_{B_1\times
		B_1^*}.
\end{align}
Also, by the definitions of $u_{\rho}^{\lambda}$ and $u_{\gamma}^{\lambda}$
\begin{align*}
	&~\mathcal{E}_{\gamma}\left(u_{\gamma}^{\lambda}\right)+\frac{\lambda}{q}\left\|u_{\gamma}^{\lambda}\right\|_{B_1}^{q}\leq\mathcal{E}_{\gamma}\left(u_{\rho}^{\lambda}\right)+\frac{\lambda}{q}\left\|u_{\rho}^{\lambda}\right\|_{B_1}^{q}\\
	\Rightarrow & ~\mathcal{E}_{\gamma}\left(u_{\gamma}^{\lambda}\right)-\mathcal{E}_{\gamma}\left(u_{\rho}^{\lambda}\right)+\frac{\lambda}{q}\left(\left\|u_{\gamma}^{\lambda}\right\|_{B_1}^{q}-\left\|u_{\rho}^{\lambda}\right\|_{B_1}^{q}\right)\leq
	0.
\end{align*}
From \eqref{eq39}, the above inequality can be written as
\begin{align*}
	0\geq ~ \Big\langle u_{\gamma}^{\lambda}-u_{\rho}^{\lambda},-p\int_{Z}\left|y-(Au_{\rho}^{\lambda})(x)\right|^{p-1}&\text{sgn}[y-(Au_{\rho}^{\lambda})(x)] H(\cdot, x) d \gamma\Big\rangle_{B_1\times B_1^*} \\
	&
	+\frac{\lambda}{q}\left(\left\|u_{\gamma}^{\lambda}\right\|_{B_1}^{q}-\left\|u_{\rho}^{\lambda}\right\|_{B_1}^{q}\right).
\end{align*}
From \eqref{eqnW} we have,
\begin{equation*}
	\left|y-(Au_{\rho}^{\lambda})(x)\right|^{p-1}\text{sgn}[y-(Au_{\rho}^{\lambda})(x)]=W_{\rho}^{\lambda}(x,y).
\end{equation*}
Then, the above equation becomes        \begin{equation*}
	0\geq ~ p\Big\langle u_{\gamma}^{\lambda}-u_{\rho}^{\lambda},-\int_{Z}W_{\rho}^{\lambda}(x,y) H(\cdot, x) d \gamma\Big\rangle_{B_1\times B_1^*}
	+\frac{\lambda}{q}\left(\left\|u_{\gamma}^{\lambda}\right\|_{B_1}^{q}-\left\|u_{\rho}^{\lambda}\right\|_{B_1}^{q}\right).
\end{equation*}

\noindent By making use of \eqref{eq6} and the definition of $j_{q}\left(u_{\rho}^{\lambda}\right)$, the above relation becomes
\begin{align*}
	0\geq &~ p\left\langle u_{\gamma}^{\lambda}-u_{\rho}^{\lambda},-\int_{Z}W_{\rho}^{\lambda}(x,y) H(\cdot, x) d \gamma\right\rangle_{B_1\times B_1^*} \\
	& \qquad+\lambda\left\langle u_{\gamma}^{\lambda}-u_{\rho}^{\lambda}, j_{q}\left(u_{\rho}^{\lambda}\right)\right\rangle_{B_1\times B_1^*}  +\lambda c_{q}\left\|u_{\gamma}^{\lambda}-u_{\rho}^{\lambda}\right\|_{B_1}^{q}
	\\  \geq ~& p\left\langle u_{\rho}^{\lambda}-u_{\gamma}^{\lambda},\int_{Z}W_{\rho}^{\lambda}(x,y) H(\cdot, x) d \gamma\right\rangle_{B_1\times B_1^*} \\
	& \qquad+\lambda\left\langle u_{\gamma}^{\lambda}-u_{\rho}^{\lambda}, j_{q}\left(u_{\rho}^{\lambda}\right)\right\rangle_{B_1\times B_1^*}  +\lambda
	c_{q}\left\|u_{\gamma}^{\lambda}-u_{\rho}^{\lambda}\right\|_{B_1}^{q}.
\end{align*}
Therefore, by Lemma \ref{lemma5},
\begin{align*}
	0~\geq & ~p\left\langle u_{\rho}^{\lambda}-u_{\gamma}^{\lambda},\int_{Z}W_{\rho}^{\lambda}(x,y) H(\cdot, x) d \gamma\right\rangle_{B_1\times B_1^*} \\
	& -p\left\langle u_{\rho}^{\lambda}-u_{\gamma}^{\lambda}, \int_{Z}W_{\rho}^{\lambda}(x,y) H(\cdot, x) d \rho\right\rangle_{B_1\times B_1^*}
	+\lambda
	c_{q}\left\|u_{\gamma}^{\lambda}-u_{\rho}^{\lambda}\right\|_{B_1}^{q}\\
	& \geq  ~p\Big\langle u_{\rho}^{\lambda}-u_{\gamma}^{\lambda},\int_{Z}W_{\rho}^{\lambda}(x,y) H(\cdot, x) d \gamma - \int_{Z}W_{\rho}^{\lambda}(x,y) H(\cdot, x) d \rho\Big\rangle_{B_1\times B_1^*}
	\\
	&+\lambda
	c_{q}\left\|u_{\gamma}^{\lambda}-u_{\rho}^{\lambda}\right\|_{B_1}^{q}.
\end{align*}
Hence, it follows that
\begin{align*}
	\lambda c_{q}\left\|u_{\gamma}^{\lambda}-u_{\rho}^{\lambda}\right\|_{B_1}^{q} \leq
	p\Big\langle u_{\gamma}^\lambda-u_{\rho}^{\lambda}, \int_{Z}W_{\rho}^{\lambda}(x,y) H(\cdot, x) d \gamma-
	\int_{Z}W_{\rho}^{\lambda}(x,y) H(\cdot, x) d \rho\Big\rangle_{B_1\times
		B_1^*}
\end{align*}
This means
\begin{align*}
	\lambda c_{q}\left\|u_{\gamma}^{\lambda}-u_{\rho}^{\lambda}\right\|_{B_1}^{q}
	\leq p\left\|u_{\gamma}^{\lambda}-u_{\rho}^{\lambda}\right\|_{B_1} \Big\| \int_{Z}W_{\rho}^{\lambda}(x,y) H(\cdot, x) d \rho -
	\int_{Z}W_{\rho}^{\lambda}(x,y) H(\cdot, x) d \gamma
	\Big\|_{B_1^{*}}.
\end{align*}

\noindent Thus,
\begin{align*}
	\left\|u_{\gamma}^{\lambda}-u_{\rho}^{\lambda}\right\|_{B_1}\leq \Big(\frac{p}{\lambda c_q} \Big\| \int_{Z}W_{\rho}^{\lambda}(x,y) H(\cdot, x) d \rho -
	\int_{Z}W_{\rho}^{\lambda}(x,y) H(\cdot, x) d \gamma
	\Big\|_{B_1^{*}}\Big)^{\frac{1}{q-1}}.
\end{align*}
This completes the proof.
\end{proof}

\begin{lemma}\label{lemma8}(cf. \cite{CuckerZhou2007}, Proposition $3.13$). Let $\mathcal{D}$ be a family of functions from a probability space $Z$ to $\mathbb R$ and $d(\cdot, \cdot)$ a metric on $\mathcal{D}$. Let $\mathcal{U} \subset$ $Z$ be of full measure such that for constants $C, L>0$,

\begin{enumerate}[(i)]
	\item $|\Psi(z)| \leq C$ on $ \mathcal{U}$ for all $\Psi \in \mathcal{D}$,\\
	
	\item $\left|L_{z}\left(\Psi_{1}\right)-L_{z}\left(\Psi_{2}\right)\right| \leq L d\left(\Psi_{1}, \Psi_{2}\right)$ for all $\Psi_{1}, \Psi_{2} \in \mathcal{D}$ and all $z \in \mathcal{U}^{m}$, where
	\begin{equation*}
		L_{z}(\Psi)=\int_{Z} \Psi(z)d\rho-\frac{1}{m} \sum_{i=1}^{m}
		\Psi\left(z_{i}\right).
	\end{equation*}
	Then $\forall\epsilon>0,$
	\begin{equation*}
		\underset{z\in Z^m}{\text{Prob}}    ~\Big\{~\underset{\Psi\in\mathcal{D} }{\text{sup}}|L_{z}(\Psi)|\leq\epsilon\Big\}\geq
		1-2\mathcal{N}\left(\mathcal{D},\frac{\epsilon}{2L}\right)\exp\Big\{-\frac{m\epsilon^2}{8C^2}\Big\}.
	\end{equation*}
\end{enumerate}
\end{lemma}
\vspace{0.1cm}

\noindent We now obtain a high probability upper bound for $\|u_{\rho}^{\lambda}-u_{z}^{\lambda}\|_{B_1}$ in terms of the covering number.
\begin{lemma}\label{lemma9}
Let $u_{\rho}^{\lambda}$ be the solution of scheme \eqref{regtrue} and $u_z^{\lambda}$ be given by \eqref{reg}.
Then, under the assumptions  $(A1)-(A3)$ and for all $\epsilon>0$  
\begin{align*}
	\underset{{z \in Z^{m}}}{\text{Prob}}\left\{\left\|u_{\rho}^{\lambda}-u_{z}^{\lambda}\right\|_{B_1} \leq \epsilon\right\}
	\geq 1-2 \mathscr{N}\left(\mathcal{G}_{\rho}^{\lambda}, \frac{\lambda c_{q} \epsilon^{q-1}}{4p}\right) \exp \left(-\frac{m \lambda^{2} c_{q}^{2} \epsilon^{2(q-1)}}{8p^2
		\Omega^{2}}\right),
\end{align*}
where $\mathcal{G}_{\rho}^{\lambda}$ is defined as in \eqref{G} and $\Omega$ as in \eqref{omega}.

\end{lemma}
\begin{proof}
Take $\gamma=\gamma_{z}$ in Lemma
\ref{lemma7}. Here, $\gamma_{z}$ is the empirical measure for which $u_z^{\lambda}=u_{\gamma}^{\lambda}$. Then from \eqref{eqnlemma51},
\begin{align}\label{eq41}
	\left\|u_{\rho}^{\lambda}-u_{z}^{\lambda}\right\|_{B_1} \notag &\leq\left(\frac{p}{\lambda c_{q}} \times \Big\| \int_{Z}W_{\rho}^{\lambda}(x,y) H(\cdot, x) d \rho\right. \\
	&\qquad\qquad \left.\quad-\frac{1}{m} \sum_{i=1}^{m}W_{\rho}^{\lambda}(x_i,y_i) H\left(\cdot, x_{i}\right) \Big\|_{B_1^{*}}\right)^{1
		/(q-1)}.
\end{align}
Let us denote $h_u(z)=h_u(x,y):=W_{\rho}^{\lambda}(x,y)(Au)(x)$ and
\noindent consider the expression
\begin{align*}
	& \Big\| \int_{Z}W_{\rho}^{\lambda}(x,y) H(\cdot, x)\, d\rho
	- {}  \frac{1}{m} \sum_{i=1}^{m} W_{\rho}^{\lambda}(x_i,y_i) H(\cdot, x_i)
	\Big\|_{B_1^{*}}\\
	& =\, \sup_{\|u\|_{B_1} \leq 1} \Big| \Big\langle u, \int_{Z} W_{\rho}^{\lambda}(x,y) H(\cdot, x)\, d\rho - \frac{1}{m} \sum_{i=1}^{m} W_{\rho}^{\lambda}(x_i,y_i) H(\cdot, x_i) \Big\rangle_{B_1\times B_1^*}
	\Big|\\
		& =\, \sup_{\|u\|_{B_1} \leq 1} \Big| \int_{Z} W_{\rho}^{\lambda}(x,y) \langle u, H(\cdot, x) \rangle_{B_1\times B_1^*}\, d\rho -
		\frac{1}{m} \sum_{i=1}^{m} W_{\rho}^{\lambda}(x_i,y_i)\langle u, H(\cdot, x_i) \rangle_{B_1\times B_1^*}
		\Big|\\
			&=\, \sup_{\|u\|_{B_1} \leq 1} \Big| \int_{Z} W_{\rho}^{\lambda}(x,y) (Au)(x)\, d\rho - \frac{1}{m} \sum_{i=1}^{m} W_{\rho}^{\lambda}(x_i,y_i) (Au)(x_i)
			\Big|\\
			& = \sup _{h_u \in \mathcal{G}_{\rho}^{\lambda}}\left|\int_{Z} h_u(z) d \rho-\frac{1}{m} \sum_{i=1}^{m}
			h_u\left(z_{i}\right)\right|.
		\end{align*}
		Thus, from \eqref{eq41} we obtain
		\begin{equation}\label{eq42}
			\frac{\lambda c_q}{p}\left\|u_{\rho}^{\lambda}-u_{z}^{\lambda}\right\|_{B_1}^{q-1} \leq \sup _{h_u \in \mathcal{G}_{\rho}^{\lambda}}\left|\int_{Z} h_u(z) d \rho-\frac{1}{m} \sum_{i=1}^{m}
			h_u\left(z_{i}\right)\right|.
		\end{equation}
		Now to use lemma \ref{lemma8}, we need to show that $h_u$ satisfies the required conditions. To show that $h_u$ is bounded, we first use 
		assumption $A(3)$ as,
		\begin{align*}
			\left|\left(y-(Au_{\rho}^{\lambda})(x)\right)\right| & \leq|y|+\left|(Au_{\rho}^{\lambda})(x)\right|\\
			&\leq M+k\|u_{\rho}^{\lambda}\|_{B_1}\\
			&\leq M+k\|u_{\rho}^{\lambda}-u_{\rho}\|_{B_1}+k\|u_{\rho}\|_{B_1}\\
			&\leq M+k\|u_{\rho}\|_{B_1}+kc_{\beta}\lambda^{\beta}.
		\end{align*}
		
		\noindent 
		Let
		$\Omega=(M+k\|u_{\rho}\|_{B_1}+kc_{\beta}\lambda^{\beta})^{p-1}k$
		and for $\n u\leq1$ consider
		\begin{align*}
			|h_u(z)|=|W_{\rho}^{\lambda}(x,y)(Au)(x)|&=|y-(Au_{\rho}^{\lambda})(x)|^{p-1}|(Au)(x)|\\
			&\leq(M+k\|u_{\rho}\|_{B_1}+kc_{\beta}\lambda^{\beta})^{p-1}k\n{u}_{B_1}\\
			&\leq(M+k\|u_{\rho}\|_{B_1}+kc_{\beta}\lambda^{\beta})^{p-1}k=\Omega.
		\end{align*}
		Observe that $\Omega$ is uniformly bounded over $\lambda$. This shows that $h_u$ is bounded.
		Let us define
		\begin{equation}\label{eq43}
			L_{z}(h_u)=\int_{Z} h_u(z) d \rho-\frac{1}{m} \sum_{i=1}^{m}
			h_u\left(z_{i}\right).
		\end{equation}
		\noindent Then, the following consequence is immediate.
		\begin{align*}
			\left|L_{z}\left(h_{u_1}\right)-L_{z}\left(h_{u_2}\right)\right|~ &\leq~\int_{Z}\left|\left(h_{u_1}(z)-h_{u_2}(z)\right) \right| d\rho+\frac{1}{m} \sum_{i=1}^{m}\Big|h_{u_1}\left(z_{i}\right)-h_{u_2}\left(z_{i}\right)\Big|  \\
			&\leq~\int_{Z}\|h_{u_1}-h_{u_2}\|_{\infty} d \rho+\frac{1}{m} \sum_{i=1}^{m}\|h_{u_1}-h_{u_2}\|_{\infty} \\
			&  \leq 2\left\|h_{u_1}-h_{u_2}\right\|_{\infty}.
		\end{align*}
		Hence both the conditions of lemma \ref{lemma8} are satisfied.
		From \eqref{eq42}, we see that
		\begin{align*}
			\left\{z \in Z^{m}:\sup _{h_u \in \mathcal{G}_{\rho}^{\lambda}}\left|L_{z}(h_u)\right| \leq  \epsilon\right\}  \subset\left\{z \in Z^{m}:\frac{\lambda c_{q}}{p}\left\|u_{\rho}^{\lambda}-u_z^{\lambda}\right\|_{B_1}^{q-1} \leq
			\epsilon\right\}.
		\end{align*}
		Therefore, it follows that
		\begin{align*}
			\underset{z\in Z^m}{\text{Prob}}\left\{\frac{\lambda c_{q}}{p}\left\|u_{\rho}^{\lambda}-u_z^{\lambda}\right\|_{{B_1}}^{q-1} \leq \epsilon\right\}
			&\geq \underset{z\in Z^m}{\text{Prob}}\left\{\sup _{h \in \mathcal{G}_{\rho}^{\lambda}}\left|L_{z}(h_u)\right| \leq
			\epsilon\right\}.
		\end{align*}
		Hence, from Lemma \ref{lemma8}, we can deduce that
		\begin{equation}
			\underset{z\in Z^m}{\text{Prob}}\left\{\frac{\lambda c_{q}}{p}\left\|u_{\rho}^{\lambda}-u_z^{\lambda}\right\|_{{B_1}}^{q-1} \leq \epsilon\right\}\geq 1-2 \mathscr{N}\left(\mathcal{G}_{\rho}^{\lambda}, \frac{\epsilon}{4}\right) \exp \left(-\frac{m \epsilon^{2}}{8
				\Omega^{2}}\right).
		\end{equation}
		
		\noindent Equivalently,
		\begin{equation*}
			\underset{z\in Z^m}{\text{Prob}}\left\{\left\|u_{\rho}^{\lambda}-u_z^{\lambda}\right\|_{{B_1}} \leq
			\left(\frac{p\epsilon}{\lambda c_q}\right)^{\frac{1}{q-1}}\right\} \geq 1-2 \mathscr{N}\left(\mathcal{G}_{\rho}^{\lambda}, \frac{\epsilon}{4}\right) \exp \left(-\frac{m \epsilon^{2}}{8
				\Omega^{2}}\right).
		\end{equation*}
		This is equivalent to
		\begin{align*}
			\underset{z\in Z^m}{\text{Prob}}\left\{\left\|u_z^{\lambda}-u_{\rho}^{\lambda}\right\|_{B_1} \leq \epsilon\right\}~\geq 1- 2 \mathcal{N}\left(\mathcal{G}_{\rho}^{\lambda}, \frac{\lambda c_{q} \epsilon^{q-1}}{4p}\right) \exp \left(-\frac{m \lambda^{2} c_{q}^{2} \epsilon^{2(q-1)}}{8p^2
				\Omega^{2}}\right).
		\end{align*}
		Hence the proof.
	\end{proof}
	
	We now prove a lemma concerning covering number.
	\begin{lemma}\label{lemma6}
		Let $K(x, y)$ be as defined in \eqref{kxy}, and suppose that $H(\cdot, x)$ is uniformly continuous about $x$ with respect to $\|\cdot\|_{{B_1^*}}$. Then, the set $\mathcal{G}_{\rho}^{\lambda}$ given by \eqref{G} is compact in $C(X\times[-M,M])$. Further
		\begin{equation}\label{union}
			\mathcal{N}(\mathcal{G}_{\rho}^{\lambda},r)\leq\mathcal{N}\left(B_K^R,r/C_{\rho}\right)
		\end{equation}
		with $C_{\rho}=(M+k\|u_{\rho}\|_{B_1})^{p-1}$
	\end{lemma}
	\begin{proof}
		For $g=Au,~u\in B_1$, the reproducing property gives
		\begin{equation*}
			|g(x)|=\left|\langle u, H(\cdot, x)\rangle_{B_1\times B_1^*}\right| \leq\|u\|_{{B_1}} \times\|H(\cdot, x)\|_{{B_1^*}}=k\times {\n g}_{B_K} .
		\end{equation*}
		Now, since $X$ is compact, $X \times X$ is also compact. Let $x,y\in X$ be such that $d(x,y)<\delta$. By the uniform continuity of $H(\cdot,x)$
		\begin{equation}
			\label{Hxeqn}
			\left\|H(\cdot, x)-H\left(\cdot, x^{\prime}\right)\right\|_{{B_1^*}}<\epsilon
		\end{equation}
		Then for $\|u\|\leq R$
		\begin{align}\label{con}
			|g(x)-g(y)|\notag
			& =\left|\langle u, H(\cdot, x)\rangle_{B_1\times B_1^*}-\left\langle u, H\left(\cdot, y\right)\right\rangle_{B_1\times B_1^*}\right| \\\notag
			& =\left|\left\langle u, H(\cdot, x)-H\left(\cdot, y\right)\right\rangle_{B_1\times B_1^*}\right|  \\
			& \leq\|u\|_{{B_1}} \times\left\|H(\cdot, x)-H\left(\cdot, y\right)\right\|_{{B_1^*}}<
			R\epsilon.
		\end{align}
		By the above assertion, the set
		\begin{equation*}
			B_K^R=\{g\in B_K:\exists u\in B_1~ \text{with}~\|u\|_{B_1}\leq R~\text{and}~g=Au\}
		\end{equation*}
		is equicontinuous. Since it is closed and bounded, $B_K^R$ is a compact subset of $C(X)$.
		\vspace{0.2cm}
		
		\noindent 
		We had from \eqref{eqnW}
		\begin{equation*}
			W_{\rho}^{\lambda}(x,y)=|y-(Au_{\rho}^{\lambda})(x)|^{p-1}\text{sgn}[y-(Au_{\rho}^{\lambda})(x)].
		\end{equation*}
		Since $(y-(Au_{\rho}^{\lambda})(x))$ is a continuous function on $X\times[-M,M]$,
		the sign term in the above equation is continuous everywhere except possibly at the points
		$\mathcal{S}=\{(x,y)\in X\times[-M,M]:y-(Au_{\rho}^{\lambda})(x)=0\}$. Therefore, $W_{\rho}^{\lambda}$ will be a continuous function on $X\times[-M,M]-\mathcal{S}$.
		
		\noindent     Now we  show that the function is also continuous on $\mathcal{S}$.
		Since the sign function is bounded and because $|y-(Au_{\rho}^{\lambda})(x)|^{p-1}\to0$ as $(x,y)\to a_0\in\mathcal{S}$,
		it follows that $W_{\rho}^{\lambda}(x,y)\to0=W_{\rho}^{\lambda}(a_0)$, that is, $W_{\rho}^{\lambda}$ is continuous on $\mathcal{S}$. Therefore,
		$W_{\rho}^{\lambda}$ is continuous everywhere on $X\times[-M,M]$.
		\noindent If we consider a map $T_W$ such that
		\begin{align*}
			T_W:~C(X)&\to C(X\times[-M,M])\\
			(T_Wf)(x,y)&=W_{\rho}^{\lambda}(x,y)f(x).
		\end{align*}
		Then, $T_W$ is a continuous operator, and hence the set $\mathcal{G}_{\rho}^{\lambda}=T_W(B_K^R)$
		given by \eqref{G} is the continuous image of a compact set. Therefore, $\mathcal{G}_{\rho}^{\lambda}$ is compact in $C(X\times[-M,M]$.
		
		We will show that for all $\lambda$, the term  $\mathcal{N}(\mathcal{G}_{\rho}^{\lambda},r)$ is uniformly bounded. First note that
		\begin{align*}
			|W_{\rho}^{\lambda}(x,y)|&=|y-(Au_{\rho}^{\lambda})(x)|^{p-1}\\
			&\leq(|y|+\left|(Au_{\rho}^{\lambda})(x)|\right)^{p-1}\\
			&\leq(M+k\|u_{\rho}^{\lambda}\|_{B_1})^{p-1}\\
			&\leq(M+k\|u_{\rho}\|_{B_1})^{p-1}=C_{\rho}
		\end{align*} 
		For some $C_{\rho}>0$. The last inequality follows from \eqref{true} and \eqref{regtrue}. Since $\mathcal{E}_{\rho}(u_{\rho})\leq\mathcal{E}_{\rho}(u_{\rho}^{\lambda})$ it is necessary that $\|u_{\rho}^{\lambda}\|_{B_1}\leq\|u_{\rho}\|_{B_1}$.\\
		
		Now we show that
		\begin{equation*}
			\mathcal{N}(\mathcal{G}_{\rho}^{\lambda},r)\leq\mathcal{N}\left(B_K^R,r/C_{\rho}\right)
		\end{equation*}
		
		Since $B_K^R$ is compact, the covering number will be finite for all $r$. Let $U(g_i,r/C_{\rho}),~i\in\{1,2,...,n\}$ be a set of open balls in $C(X)$ of radius $r/C_{\rho}$ centered at $g_i$ that covers $B_K^R$. Then we claim that
		\begin{equation}\label{cover}
			\bigcup_{i=1}^n U(T_W(g_i),r)
		\end{equation} 
		will cover $\mathcal{G}_{\rho}^{\lambda}$. Let $f\in\mathcal{G}_{\rho}^{\lambda}$, then $f=T_W(g)$ for some $g\in B_K^R$. We have $g\in U(g_i,r/C_{\rho})$ for some $i$ with $\|g-g_i\|_{\infty}<r/C_{\rho}$. This implies that
		\begin{align*}
			\|T_W(g)-T_W(g_i)\|_{\infty}&=\|W_{\rho}^{\lambda}(x,y)g(x)-W_{\rho}^{\lambda}(x,y)g_i(x)\|_{\infty}\\
			&\leq \|W_{\rho}^{\lambda}(x,y)\|_{\infty}\|g-g_i\|_{\infty}< r
		\end{align*}
		That is $f=T_W(g)\in U(T_W(g_i),r)$. Since $f$ is arbitrary, our statement \eqref{cover} holds true. Therefore, \eqref{union} is proved.
	\end{proof}
	
	\begin{lemma}\label{lemma10}
		(cf. \cite{CuckerSmale2002}). Let $a>b>0$ and $u_{1}>0, u_{2}>0$. Then, the equation given by
		\begin{equation*}
			x^{a}-u_{1} x^{b}-u_{2}=0
		\end{equation*}
		has a unique positive solution $l$. Also
		\begin{equation*}
			l \leq \max \left\{\left(2 u_{1}\right)^{\frac{1}{a-b}},\left(2
			u_{2}\right)^{\frac{1}{a}}\right\}.
		\end{equation*}
	\end{lemma}
	\vspace{0.4cm}
	\noindent We are now ready to present the proof of theorem \ref{thm1}.
	
	\begin{proof}[\textbf{Proof of theorem}     \ref{thm1}]
		From the triangle inequality and \eqref{eq9},
		\begin{align*}
			\n{u_z^{\lambda}-\p u}_{B_1}&\leq\n{u_z^{\lambda}-u_{\rho}^{\lambda}}_{B_1}+\n{u_{\rho}^{\lambda}-\p u}_{B_1}\\
			\n{u_z^{\lambda}-\p
				u}_{B_1}&\leq\n{u_z^{\lambda}-u_{\rho}^{\lambda}}_{B_1}+c_{\beta}{\lambda}^{\beta}.
		\end{align*}
		This gives us
		\begin{equation*}
			\left(\left\|u_z^{\lambda}-u_{\rho}\right\|_{B_1}-c_{\beta}{\lambda}^{\beta}\right)
			\leq\left\|u_z^{\lambda}-u_{\rho}^{\lambda}\right\|_{B_1}.
		\end{equation*}
		This implies that for any $\epsilon>0$,
		\begin{align*}
			\left\{z \in Z^{m}:\left\|u_z^{\lambda}-u_{\rho}^{\lambda}\right\|_{B_1}\leq\epsilon\right\}
			\subset \{z  \left.\in Z^{m}:
			\left(\left\|u_z^{\lambda}-u_{\rho}\right\|_{B_1}-c_{\beta}{\lambda}^{\beta}\right)\leq\epsilon\right\}.
		\end{align*}
		By Lemma \ref{lemma9} and from the above expression, we conclude that
		\begin{align*}
			\underset{z\in Z^m}{\text{Prob}}\left\{\left(\left\|u_z^{\lambda}-u_{\rho}\right\|_{B_1}-c_{\beta}{\lambda}^{\beta}\right)\leq\epsilon\right\} \geq     \underset{z\in Z^m}{\text{Prob}}\left\{\left\|u_z^{\lambda}-u_{\rho}^{\lambda}\right\|_{B_1}\leq\epsilon\right\}
		\end{align*}
		\begin{align*}
			~\geq 1-2 \mathscr{N}\left(\mathcal{G}_{\rho}^{\lambda}, \frac{\lambda c_{q} \epsilon^{q-1}}{4p}\right) \exp \left(-\frac{m \lambda^{2} c_{q}^{2} \epsilon^{2(q-1)}}{8p^2
				\Omega^{2}}\right).
		\end{align*}
		Equivalently,
		\begin{align}\label{eq45}
			\underset{z\in Z^m}{\text{Prob}} \notag& \left\{\left\|u_z^{\lambda}-u_{\rho}\right\|_{B_1} \leq  \epsilon+c_{\beta}{\lambda}^{\beta}\right\} \\
			& \geq 1-2 \mathcal{N}\left(\mathcal{G}_{\rho}^{\lambda}, \frac{\lambda c_{q} \epsilon^{q-1}}{4p}\right) \times \exp \left(-\frac{m \lambda^{2} c_{q}^{2} \epsilon^{2(q-1)}}{8p^2
				\Omega^{2}}\right).
		\end{align}
		By Lemma \ref{lemma6} and equation \eqref{G}, the covering number of $\mathcal{G}_{\rho}^{\lambda}$
		is uniformly bounded. By assumptions of the theorem \eqref{thm1}, this covering number exhibits a logarithmic
		complexity exponent $s \geq 0$, that is,
		\begin{equation*}
			\mathcal{N}\left(\mathcal{G}_{\rho}^{\lambda}, \frac{\lambda c_{q} \epsilon^{q-1}}{4p}\right) \leq \exp\left(c_{s}\left(\frac{4p \Omega}{\lambda c_{q}
				\epsilon^{q-1}}\right)^{s}\right).
		\end{equation*}
		Then, by \eqref{eq45}
		\begin{align*}
			& \underset{z\in Z^m}{\text{Prob}}\left\{\left\|u_z^{\lambda}-u_{\rho}\right\|_{B_1} \leq  \epsilon+c_{\beta}{\lambda}^{\beta}\right\} \\
			& \quad \geq 1-2 \exp \left(c_{s}\left(\frac{4p \Omega}{\lambda c_{q} \epsilon^{q-1}}\right)^{s}-\frac{m \lambda^{2} c_{q}^{2} \epsilon^{2(q-1)}}{8p^2
				\Omega^{2}}\right).
		\end{align*}
		If we take
		\begin{equation*}
			2 \exp \left(c_{s}\left(\frac{4p \Omega}{\lambda c_{q} \epsilon^{q-1}}\right)^{s}-\frac{m \lambda^{2} c_{q}^{2} \epsilon^{2(q-1)}}{8p^2
				\Omega^{2}}\right)=\delta,
		\end{equation*}
		then we will have the equation given by
		\begin{align*}
			\epsilon^{(2+s)(q-1)}-\frac{8p^2 \Omega^{2} \log (2 / \delta)}{m \lambda^{2} c_{q}^{2}} \epsilon^{s(q-1)}
			-\frac{8p^2 \Omega^{2} c_{s}}{m \lambda^{2} c_{q}^{2}} \times\left(\frac{4p \Omega}{\lambda
				c_{q}}\right)^{s}=0.
		\end{align*}
		\vspace{0.3cm}
		\noindent By Lemma \ref{lemma10}, the unique positive solution $\epsilon$ has the bound
		\begin{align}\label{eq46}
			\epsilon \notag& \leq \max \left[\left(\frac{16p^2 \Omega^{2} \log (2 / \delta)}{m \lambda^{2} c_{q}^{2}}\right)^{1 / 2(q-1)},\left(\frac{16p^2 \Omega^{2} c_{s}}{m \lambda^{2} c_{q}^{2}} \times\left(\frac{4p \Omega}{\lambda c_{q}}\right)^{s}\right)^{1 /(2+s)(q-1)}\right]\nonumber  \\\notag
			\\&\leq\max \left[\sqrt[q-1]{\frac{4p\Omega}{\lambda c_q}}\left(\frac{ \log (2 / \delta)}{m}\right)^{1 / 2(q-1)},\sqrt[q-1]{\frac{4p\Omega}{\lambda c_q}}\left(\frac{c_{s}}{m} \right)^{1
				/(2+s)(q-1)}\right]\nonumber\\
			\\&\leq\sqrt[q-1]{\frac{4p\Omega}{\lambda c_q}}~\max\left[\left(\frac{ \log (2 / \delta)}{m}\right)^{1 / 2(q-1)},\left(\frac{c_{s}}{m} \right)^{1 /(2+s)(q-1)}\right].\\\notag
		\end{align}
		By \eqref{eq45} and \eqref{eq46} with a probability of at least $1-\delta$, we can state that
		
		\begin{align*}
			\left\|u_z^{\lambda}-u_{\rho}\right\|_{B_1} \leq  \sqrt[q-1]{\frac{4p\Omega}{\lambda c_q}}~\max\left[\left(\frac{ \log (2 / \delta)}{m}\right)^{\frac{1}{2(q-1)}},\left(\frac{c_{s}}{m}
			\right)^{\frac{1}{(2+s)(q-1)}}\right]+c_{\beta}{\lambda}^{\beta}.
		\end{align*}
		This completes the proof of the theorem.
		
	\end{proof}
	
	\noindent In order to prove the next important result,  we need the following result.
	\begin{lemma} \label{lemma11}
		Let $n_1$ and $n_2$ be two positive real numbers such that $n_1<n_2$. Then, for every $D>0$, one can find a natural number $m_D$ such that for all integers $m>m_D$, we have
		\begin{equation}\label{eq47}
			\frac{D}{m^{\frac{1}{n_1}}}<\frac{1}{m^{\frac{1}{n_2}}}.
		\end{equation}
	\end{lemma}
	\begin{proof}
		Since $n_1<n_2$, we have for all $m$ that
		\[\frac{1}{m^{\frac{1}{n_1}}}\leq\frac{1}{m^{\frac{1}{n_2}}}.\]
		Fix some $m_0\in\mathbb{N}$. If $D$ is small enough such that
		\[\frac{D}{m_0^{\frac{1}{n_1}}}<\frac{1}{m_0^{\frac{1}{n_2}}},\]
		then the result will be true by taking $m_D=m_0$. Therefore, let us consider the case when $D$ is large such that
		\[\frac{D}{m^{\frac{1}{n_1}}_0}\geq\frac{1}{m^{\frac{1}{n_2}}_0}.\]
		Let $m_D$ be the smallest integer for which
		\begin{equation}\label{eq48}
			m_D>D^{\frac{1}{\frac{1}{n_1}-\frac{1}{n_2}}}.
		\end{equation}
		Since $n_1<n_2$, the exponent in the RHS of \eqref{eq48} is strictly positive, which gives
		\begin{equation*}
			m_D^{\frac{1}{n_1}-\frac{1}{n_2}}>D.
		\end{equation*}
		This gives us
		\[\frac{D}{m^{\frac{1}{n_1}}_D}<\frac{1}{m^{\frac{1}{n_2}}_D}.\]
		Then, equation \eqref{eq47} follows for all $m\geq m_D$. Hence the proof.\\
	\end{proof}
	\noindent Finally we give the proof of theorem \ref{thm2}
	\begin{proof}[\textbf{Proof of theorem} \ref{thm2}]
		Let $D(\delta):=\left(\frac{\log (2 / \delta)}{c_s^{\frac{2}{2+s}}}\right)^{\frac{1}{2(q-1)}}>0$ and consider the expression
		\begin{align}\label{eq49}
			&¬\max\left[\left(\frac{ \log (2 / \delta)}{m}\right)^{\frac{1}{2(q-1)}},\left(\frac{c_{s}}{m}\notag \right)^{\frac{1}{(2+s)(q-1)}}\right]\\
			&=c_s^{\frac{1}{(2+s)(q-1)}}\max\left[\left(\frac{\log (2 / \delta)}{c_s^{\frac{2}{2+s}}}\right)^{\frac{1}{2(q-1)}}\left(\frac{1}{m}\right)^{\frac{1}{2(q-1)}},\left(\frac{1}{m} \right)^{\frac{1}{(2+s)(q-1)}}\right]\\
			&=c_s^{\frac{1}{(2+s)(q-1)}}\max\left[D(\delta)\left(\frac{1}{m}\right)^{\frac{1}{2(q-1)}},\left(\frac{1}{m}
			\right)^{\frac{1}{(2+s)(q-1)}}\right].\nonumber
		\end{align}
	
	\noindent   In addition, for all $m\in\mathbb{N}$
	\[\left(\frac{1}{m}\right)^{\frac{1}{2(q-1)}}<\left(\frac{1}{m}\right)^{\frac{1}{(2+s)(q-1)}}\]
	Note that $\delta\to0$ implies $D(\delta)\to\infty$.
	Now, from Lemma \ref{lemma11}, for every $D(\delta)$, there is an $m_{\delta}\in\mathbb{N}$ such that for all $m>m_{\delta}$, it holds that
	\[D(\delta)\left(\frac{1}{m}\right)^{\frac{1}{2(q-1)}}<\left(\frac{1}{m} \right)^{\frac{1}{(2+s)(q-1)}}.\]
	That is,
	\begin{equation}\label{eq50}
		\max\left[D(\delta)\left(\frac{1}{m}\right)^{\frac{1}{2(q-1)}},\left(\frac{1}{m} \right)^{\frac{1}{(2+s)(q-1)}}\right]=\left(\frac{1}{m}
		\right)^{\frac{1}{(2+s)(q-1)}}.
	\end{equation}
	Hence, from \eqref{eq22}, \eqref{eq49} and \eqref{eq50}, we have for sufficiently large $m$ such that
	\begin{equation}
		\underset{z\in Z^m}{\text{Prob}}\left\{\left\|u_z^{\lambda}-u_{\rho}\right\|_{B_1} \leq  \sqrt[q-1]{\frac{4p\Omega}{\lambda c_q}}\times\sqrt[(2+s)(q-1)]{\frac{c_{s}}{m}}+c_{\beta}{\lambda}^{\beta}\right\}\geq
		1-\delta.
	\end{equation}
	This completes the proof of Theorem \ref{thm2}.\\
\end{proof}
\noindent We note that when $\lambda$ becomes small,  the term $\Omega$ tends to $(M+k\|u_{\rho}\|_{B_1})^{p-1}k$, a constant.
We can therefore now prove corollary \ref{thm3}.

\begin{proof}[\textbf{Proof of corollary} \ref{thm3}]
	In order to prove the result, let us take
	\begin{equation*}\label{eq52}
		\lambda={m^{-\frac{1}{(2+s)(1+\beta(q-1))}}}.
	\end{equation*}
	For simplicity, let's use the following notation
	\begin{equation*}
		\sqrt[q-1]{\frac{4p}{c_q}}\times\sqrt[(2+s)(q-1)]{{c_{s}}}=a'.
	\end{equation*}
	Then, the RHS of \eqref{eqthm2} in Theorem \ref{thm2} takes the form
	\begin{align*}
		a'\sqrt[q-1]{\frac{\Omega}{\lambda}}\times\sqrt[(2+s)(q-1)]{\frac{1}{m}}+c_{\beta}{\lambda}^{\beta}.
	\end{align*}
	
	\noindent Substitute for $\lambda$ as given in \eqref{lambda} to obtain the form
	\begin{align*}
		a'\sqrt[q-1]{\Omega{m^{\frac{1}{(2+s)(1+\beta(q-1))}}}}\times\sqrt[(2+s)(q-1)]{\frac{1}{m}}+c_{\beta}{{m^{-\frac{\beta}{(2+s)(1+\beta(q-1))}}}}.
	\end{align*}
	
	\noindent  On simplifying the terms in the exponent of $m$, we get
	\begin{align*}
		m^{-\frac{\beta}{(2+s)[1+\beta(q-1)]}}\left(a'\Omega^{\frac{1}{q-1}}+c_{\beta}\right).
	\end{align*}
	
	\noindent  Therefore, from \eqref{eqthm2} of Theorem \ref{thm2}, we get the inequality
	\begin{equation*}
		\underset{z\in Z^m}{\text{Prob}}\left\{\left\|u_z^{\lambda}-u_{\rho}\right\|_{B_1} \leq m^{-\frac{\beta}{(2+s)[1+\beta(q-1)]}}\left(a'\Omega^{\frac{1}{q-1}}+c_{\beta}\right)\right\}\geq
		1-\delta.
	\end{equation*}
	Observe that from \eqref{omega}, if $m\to\infty$, then $\lambda\to0$, so $\Omega^{\frac{1}{q-1}}$, and hence $a'\Omega^{\frac{1}{q-1}}+c_{\beta}$ tends to a constant, say $b'$,
	and accordingly, the convergence will be of the order
	\begin{equation*}
		\underset{z\in Z^m}{\text{Prob}}\left\{\left\|u_z^{\lambda}-u_{\rho}\right\|_{B_1} \leq b'm^{-\frac{\beta}{(2+s)[1+\beta(q-1)]}}\right\}\geq
		1-\delta.
	\end{equation*}
	
	\noindent    That is, for $\delta\in(0,1)$  with confidence $1-\delta$, as $m\to\infty$, we have the asymptotic rate of convergence of the order
	\begin{equation*}
		\left\|u_z^{\lambda}-u_{\rho}\right\|_{B_1} =
		O(m^{-\frac{\beta}{(2+s)[1+\beta(q-1)]}})= O(\lambda^{\beta}).
	\end{equation*}
	Hence the proof.
\end{proof}

\section{Numerical illustration}
\label{sec4}
In this section, we consider an example that illustrates the
theoretical results explained in the previous section. Let us
consider the space $l^3_n$, the set of all real sequences
$a=(a_n)_{n\in\mathbb{N}}$ with
$\|a\|_{l^3_n}^3=\sum_{n=1}^{\infty}{n^4|a_n|^3}<\infty$. Under the map
$D_n:l^3_n\to l^3$, defined by
$D_n(a_1,a_2,\cdots)=(a_1,2^{4/3}a_2,3^{4/3}a_3\cdots)$, the
spaces $l^3_n $ and $l^3$ are isometrically isomorphic and hence
$l^3_n$ is a $3$- uniform convex space. Let for all
$i\in\mathbb{N}$, $\phi_i\in C[-1,1]$ denote the
${(i-1)}^{\text{th}}$ degree Legendre polynomial.
Assume each $\phi_i$ is standardized, that is, $|\phi_i(x)|\leq1$ on $[-1,1]$. Let us take  $p=2.$\\


\noindent Define a map $A$ from $l^3_n$ to $L^2[-1,1]$ given by
\begin{equation}
	\label{eqnnum1}
	A(a)=\sum_{i=1}^{\infty}a_i\phi_i
\end{equation}
where $a=(a_1,a_2,\dots)$. Then by the Holder's inequality, the RHS of \eqref{eqnnum1} is well defined and continuous.
Also, the given operator is an injective  bounded linear operator. To show that it is ill-posed,
take $a=(0,0,\dots,\frac{1}{n},0,\dots)$, where $\frac{1}{n}$ is placed in the $n^{\text{th}}$ coordinate. Then, the image of $a$ is
\[A(a)=\frac{1}{n}\phi_{n}.\]
Since $A(0)=0$, we have the following consequence
\[\|A(0)-A(a)\|_{L^2}=\left\|\frac{1}{n}\phi_n\right\|_{L^2}\leq\left|\frac{1}{n}\right|\to0\quad\text{as}\quad n\to\infty.\]
But we have $\|a-0\|_{l^3_n}=n^{\frac{1}{3}}\to\infty$ as $n\to\infty$.
\vspace{0.3cm}

\noindent Now if $f\in A(l^3_n)$, then we have $f=\sum_{i=1}^{\infty}u_i\phi_i$
for some $u=(u_1,u_2,\dots)\in l^3_n$. Define $B_K$ as in
definition \ref{def1}, then the norm of $f$ becomes
\[\|f\|_{B_K}=\|u\|_{l^3_n}=\left(\sum_{i=1}^{\infty}i^4|u_i|^3\right)^{\frac{1}{3}}.\]

\noindent The continuous dual $B_K^*$ is isometric to $l^{{3}/{2}}_n$. Then, the space $B_K^*$ can be identified with $l^{{3}/{2}}_n$ as
\begin{equation*}
	B_K^*=\Big\{g^*=\sum_{i=1}^{\infty}b_i\phi_i~|~b=(b_1,b_2,\dots)\in l^{{3}/{2}}_n\Big\}~~\text{such that}~~\|g^*\|_{B_K^*}=\|b\|_{l^{{3}/{2}}_n}
\end{equation*}
with the associated duality product for $f=\sum_{i=1}^{\infty}u_i\phi_i$ and $g^*=\sum_{i=1}^{\infty}b_i\phi_i$ as
\begin{equation*}
	\langle f,g^*\rangle_{B_K\times B_K^*}=\langle u,b\rangle_{l^3_n\times
		l^{3/2}_n}=\sum_{i=1}^{\infty}{i^4u_ib_i}.
\end{equation*}
The above product is equivalent to the condition that
\begin{equation*}
	\Big\langle \sum_{i=1}^{\infty} u_i\phi_i,\sum_{j=1}^{\infty}b_j\phi_j \Big\rangle_{B_K\times B_K^*}=\sum_{i,j=1}^{\infty}u_ib_j{\langle \phi_i,\phi_j\rangle}_{B_K\times
		B_K^*}
\end{equation*}
with product between $\phi_i$ and $\phi_j$ defined as
\begin{equation*}
	{\langle \phi_i,\phi_j\rangle}_{B_K\times
		B_K^*}=\frac{2i+1}{2}i^4\int_{-1}^1\phi_i(x)\phi_j(x)dx = \left\{\begin{array}{ll}
		0 &\quad\hbox{if}\quad i\neq j\\
		i^4 &\quad\hbox{otherwise}.
	\end{array}
	\right.
\end{equation*}
\vspace{0.2cm}

\noindent Then, $B_K$ will be an RKBS with the kernel function
\begin{equation}
	K(x,y)=\sum_{i=1}^{\infty}\frac{1}{i^4}\phi_i(x)\phi_i(y).
\end{equation}
To show that consider
\begin{equation*}
	H(\cdot,y)=A^*K(\cdot,y)=\sum_{i=1}^{\infty}\frac{1}{i^4}\phi_i(y)\phi_i.
\end{equation*}
Then, it is evident that for $f=\sum_{i=1}^{\infty}u_i\phi_i=Au$,
\begin{align*}
	\langle u,H(\cdot,y)\rangle_{{l^3_n\times l^{3/2}_n}}=\sum_{i=1}^{\infty}u_i\phi_i(y)=f(y)
\end{align*}
as required. Now, since $l^{3/2}_n$ is isometric to $B_K^*$ and
$$|(Au)(y)|\leq\sum_{i=1}^{\infty}|u_i|=\sum_{i=1}^{\infty}i^{4/3}i^{-4/3}|u_i|\leq
\left(\sum_{i=1}^{\infty}i^{-2}\right)^{1/2}\|u\|_{l^3_n}$$ both the
assumptions $A(1)$ and $A(2)$ are satisfied.\\

\noindent We simulate the data points as follows.   Let
$f_{\rho}(x)=1 + x -x^2+ 0.5x^7$ and consider the data points
$z=\{(x_i,y_i)\}_{i=1}^{m}$.
Our distribution $\rho(x,y)$ is given such that the marginal distribution $\rho_X(x)$ is the uniform distribution on $[-1,1]$ and the conditional distribution $\rho(y|x)$ is
\begin{equation*}
	\rho(y|x)=f_{\rho}(x)+ N(0,\sigma^2),
\end{equation*}
where $N(0,\sigma^2)$ is the Guassian distribution with mean $0$ and variance $\sigma^2$. Then,
it follows from \cite{CuckerZhou2007} that for $p=2$,
\begin{equation*}
	u_{\rho}=\underset{u\in B_1}{\arg \min}\int_Z (y-(Au)(x))^2 d\rho(x,y)
\end{equation*}
and for all $u\in l^3_n$
\begin{equation*}
	\|Au_{\rho}^{\lambda}-Au_{\rho}\|_{L^2}^2+\frac{\lambda}{3}\|u_{\rho}^{\lambda}\|^3_{l^3_n}\leq
	\|Au-Au_{\rho}\|_{L^2}^2+\frac{\lambda}{3}\|u\|_{l^3_n}^3.
\end{equation*}
In particular for $u=u_{\rho}$,
\begin{equation*}\label{ineq1}
	\|Au_{\rho}^{\lambda}-Au_{\rho}\|_{L^2}^2+\frac{\lambda}{3}\|u_{\rho}^{\lambda}\|_{l^3_n}^3\leq\frac{\lambda}{3}\|u_{\rho}\|_{l^3_n}^3
\end{equation*}
\begin{equation*}
	\frac{1}{3}\|u_{\rho}^{\lambda}\|_{l^3_n}^3-\frac{1}{3}\|u_{\rho}\|_{l^3_n}^3\leq-\frac{1}{\lambda}\|Au_{\rho}^{\lambda}-Au_{\rho}\|_{L^2}^2.
\end{equation*}
Now we claim that the assumption $A(3)$ holds whenever there exists
$\omega\in L^2[-1,1]$ such that $A^*\omega\in J_q(u_{\rho})$. This
condition is common in inverse problem literature and is satisfied for our $f_{\rho},$ as it involves only a
finite number of basis elements. From \eqref{eq6}, we obtain
\begin{equation*}
	c_q\|u_{\rho}^{\lambda}-u_{\rho}\|^3_{l^3_n}\leq-\frac{1}{\lambda}\|Au_{\rho}^{\lambda}-Au_{\rho}\|_{L^2}^2-\langle u_{\rho}^{\lambda}-u_{\rho},j_q(u_{\rho})\rangle_{l^3_n\times l^{3/2}_n}
\end{equation*}
\begin{equation*}
	c_q\|u_{\rho}^{\lambda}-u_{\rho}\|^3_{l^3_n}\leq-\frac{1}{\lambda}\|Au_{\rho}^{\lambda}-Au_{\rho}\|_{L^2}^2-\langle A(u_{\rho}^{\lambda}-u_{\rho}),\omega\rangle_{L^2\times
		L^2}.
\end{equation*}
As $-\langle
A(u_{\rho}^{\lambda}-u_{\rho}),\omega\rangle_{L^2\times
	L^2}=\langle
A(u_{\rho}-u_{\rho}^{\lambda}),\omega\rangle_{L^2\times L^2}\leq
\|A(u_{\rho}-u_{\rho}^{\lambda})\|_{L^2}\|\omega\|_{L^2},$ we get
\begin{equation}
	\label{eqnnum2}
	c_q\|u_{\rho}^{\lambda}-u_{\rho}\|^3_{l^3_n}\leq-\frac{1}{\lambda}\|Au_{\rho}^{\lambda}-Au_{\rho}\|_{L^2}^2+\|A(u_{\rho}-u_{\rho}^{\lambda})\|_{L^2}\|\omega\|_{L^2}.
\end{equation}
Since the RHS of \eqref{eqnnum2} is a quadratic expression, whose leading term
is negative, we get the following bound
\begin{equation*}
	c_q\|u_{\rho}^{\lambda}-u_{\rho}\|^3_{l^3_n}\leq
	\frac{\lambda\|\omega\|^2_{L^2}}{4}.
\end{equation*}
i.e.,
\begin{equation}
	\|u_{{\rho}}^{\lambda}-u_{{\rho}}\|_{l^3_n}\leq
	\frac{\lambda^{1/3}}{\sqrt[3]{4c_q}}\|\omega\|_{L^2}^{\frac{2}{3}}=c_{\beta}\lambda^{\beta}.
\end{equation}
Hence, assumption $A(3)$ is satisfied with $\beta=\frac{1}{3}$.
Note that the kernel is a continuous function on a compact domain; therefore it is uniformly continuous as required by our theorem.

\noindent    Finally, to show the logarithmic complexity, let
\begin{equation*}
	B_K^1=\{f\in B_K:~\|f\|_{B_K}\leq1\}
\end{equation*}
This means that for all $f=\sum_{i=1}^{\infty}u_i\phi_i$, we have $|u_i|\leq\frac{1}{i^{4/3}}$.
Let $f_n$ denote $f$ truncated to the first $n$ coordinates. Then, for all $f\in B_K^1$ $$\|f-f_n\|_{\infty}=\|\sum_{i> n}u_i\phi_i\|_{\infty}\leq\sum_{i> n}\frac{1}{i^{4/3}}.$$
Choose $n$ such that the last sum is less than or equal to  $\frac{\epsilon}{2}$. This can be done by selecting $n$ which will satisfy the following inequality
\begin{equation*}
	\sum_{i> n}^{\infty}\frac{1}{i^{4/3}}\leq\int_{n-1}^{\infty}\frac{1}{x^{4/3}}~dx\leq\frac{\epsilon}{2}
\end{equation*}
which is true when $n\geq 1+\frac{216}{\epsilon^3}$. Let
\begin{equation}\label{neq1}
	n=\left\lfloor 1+\frac{216}{\epsilon^3}\right\rfloor+1\leq
	1+\frac{216}{\epsilon^3}+1.
\end{equation}
Then, we will consider $B_K^{1,n}$, the set of elements in $B_K^1$ truncated to the first $n$ coordinates. Thus, $B_K^{1,n}$ is a bounded subset of a finite dimensional space.
Therefore, it will give us
\begin{equation*}
	\mathcal{N}(B_K^{1,n},\epsilon)\leq
	C\left(\frac{1}{\epsilon}\right)^n.
\end{equation*}
Let $U{(f_i,\frac{\epsilon}{2})}, i=1,\dots,l$ be the finite set of open balls of radius $\frac{\epsilon}{2}$ that covers $B_K^{1,n}$.
If we extend these balls to radius $\epsilon$, then for any $f\in B_K^1$, there exists some $f_i$
such that $\|f-f_i\|_{\infty}\leq \|f-f_n\|_{\infty}+\|f_n-f_i\|_{\infty}< \epsilon$.
That is, $f\in U{(f_i,{\epsilon})}$. Hence, $B_K^1$ has covering number of the form
$$\mathcal{N}(B_K^1,\epsilon)\leq\mathcal{N}(B_K^{1,n},\frac{\epsilon}{2})\leq C\left(\frac{2}{\epsilon}\right)^n.$$
This gives us the following bound
\begin{align*}
	\ln\mathcal{N}(B_K^1,\epsilon)\leq Cn\ln\left(\frac{2}{\epsilon}\right)\leq
	Cn\left(\frac{2}{\epsilon}\right).
\end{align*}
Since our $n$ has the form given by \eqref{neq1}, we get
\begin{align*}
	\frac{n}{\epsilon}\leq\frac{2}{\epsilon}+\frac{216}{\epsilon^4}
	\leq\frac{2}{\epsilon^4}+\frac{216}{\epsilon^4}.
\end{align*}
Finally, we can state for some $C_2>0$ that
\begin{align*}
	\ln\mathcal{N}(B_K^1,\epsilon)\leq
	C_2\left(\frac{1}{\epsilon^4}\right).
\end{align*}
Since our set $\mathcal{G}_{\rho}^{\lambda}$ from \eqref{G} is a continuous image of $B_K^1$, our assumption is satisfied.\\

\noindent We take $\sigma=0.1, 0.05$ and $0.01$ for analysis purposes and try to obtain the error $\|u_z^{\lambda}-u_{\rho}\|_{l^3_n}$ for our given problem.
The true solution is given by $
u_{\rho}=\left({2}/{3},{291}/{256},-{2}/{3},{35}/{96},0,{7}/{32},0,{1}/{32},0,\dots\right)$.
The unregularized and regularized solutions are plotted against
the true solution in Figure \ref{fig1}  for
$m=50,100,200$ when $\sigma=0.1$. The respective computed solutions when $\sigma=0.05$ and $\sigma=0.01$ are given
in Figure \ref{fig2} and Figure \ref{fig3}.
The true solution $u_{\rho}$ is given by the dashed curve.
\\
\begin{figure}[H]
	\subfloat[$\sigma=0.1, m=50,100,200$ ]{\includegraphics[width
		= 2.75in, height=1.7in]{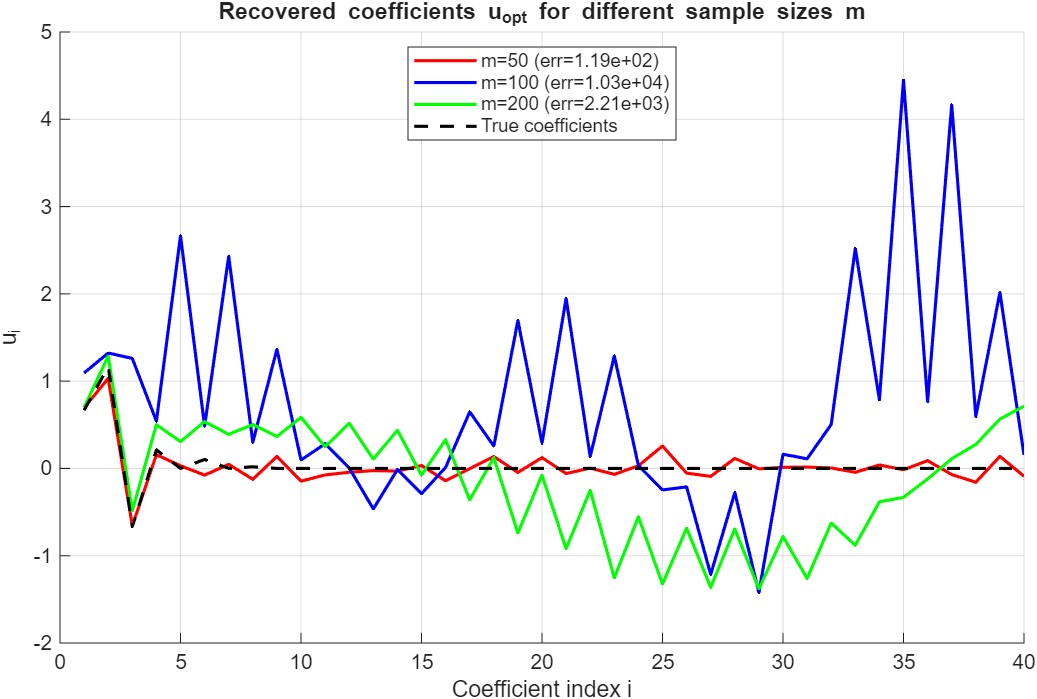}}
	\subfloat[$\sigma=0.1,m= 50,100,200$]{\includegraphics[width =
		2.75in, height=1.7in]{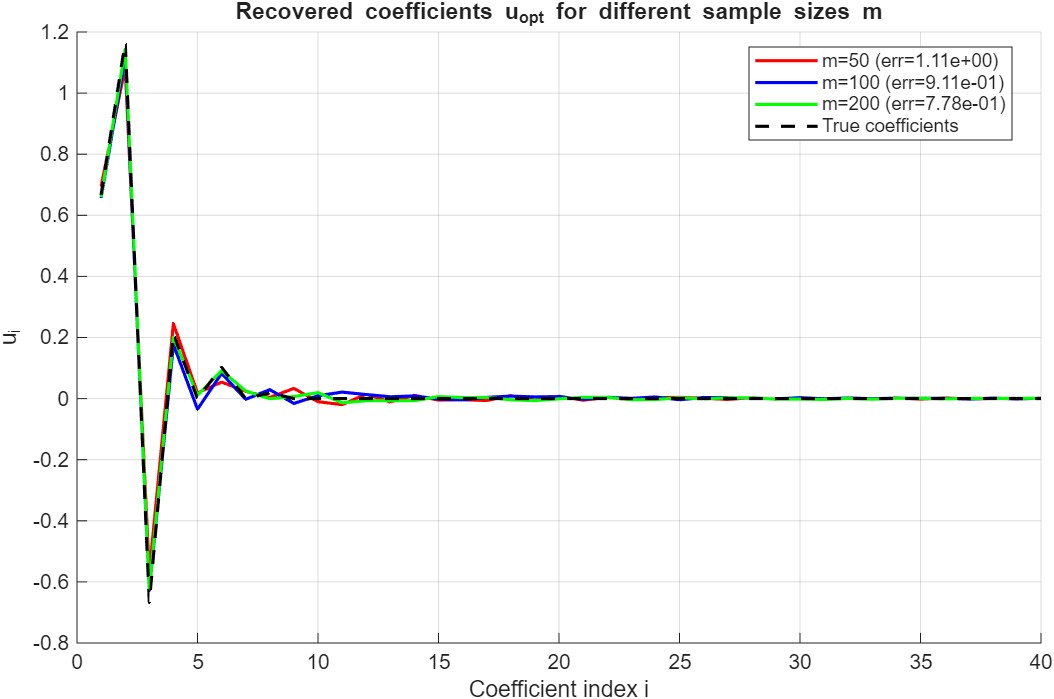}}
	\caption{Computed regularized solutions for $\sigma=0.1$ and different $m$ } \label{fig1}
\end{figure}

%
\noindent The error estimates,
$\|u_z^{\lambda}-u_{\rho}\|_{l^3_n},$ for different values of $m$
and different values of $\sigma$ in  regularized and unregularized
cases are summarized in Table \ref{table1}. 
Numerical results assert that the
regularized solution works better than the unregularized solution.

\begin{figure}[H]
	\subfloat[$\sigma=0.05, m=50,100,200$
	]{\includegraphics[width = 2.75in,
		height=1.7in]{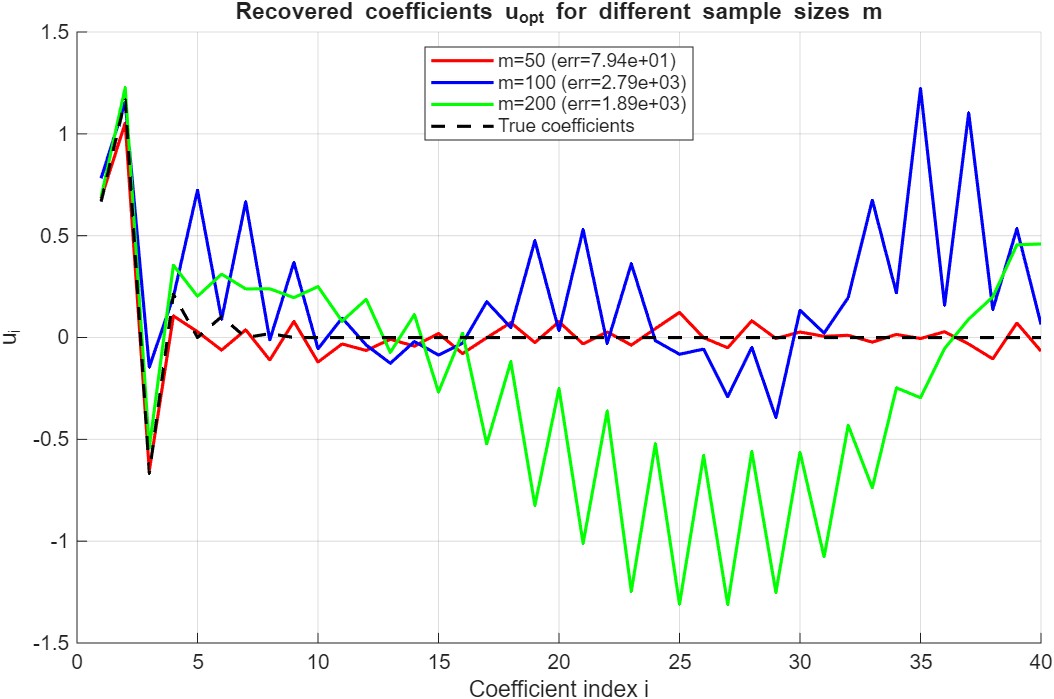}} \subfloat[$\sigma=0.05,m=
	50,100,200$]{\includegraphics[width = 2.75in,
		height=1.7in]{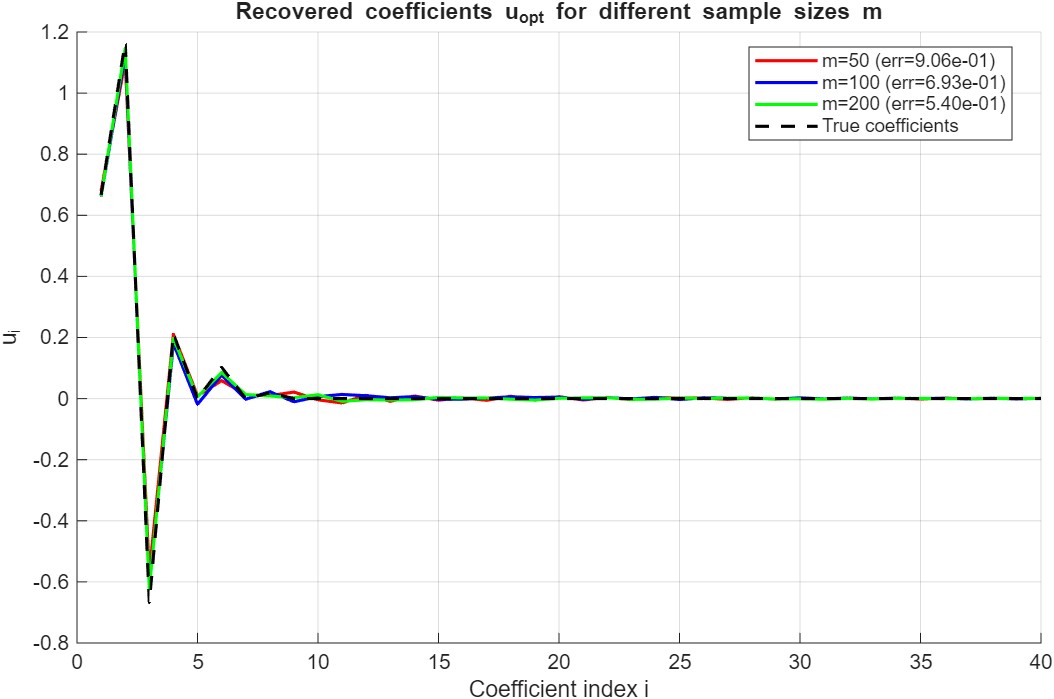}}
	\caption{Computed regularized solutions for $\sigma=0.05$ and different $m$ } \label{fig2}
\end{figure}

\begin{figure}[H]
	\subfloat[$\sigma=0.01, m=50,100,200$
	]{\includegraphics[width = 2.75in,
		height=1.7in]{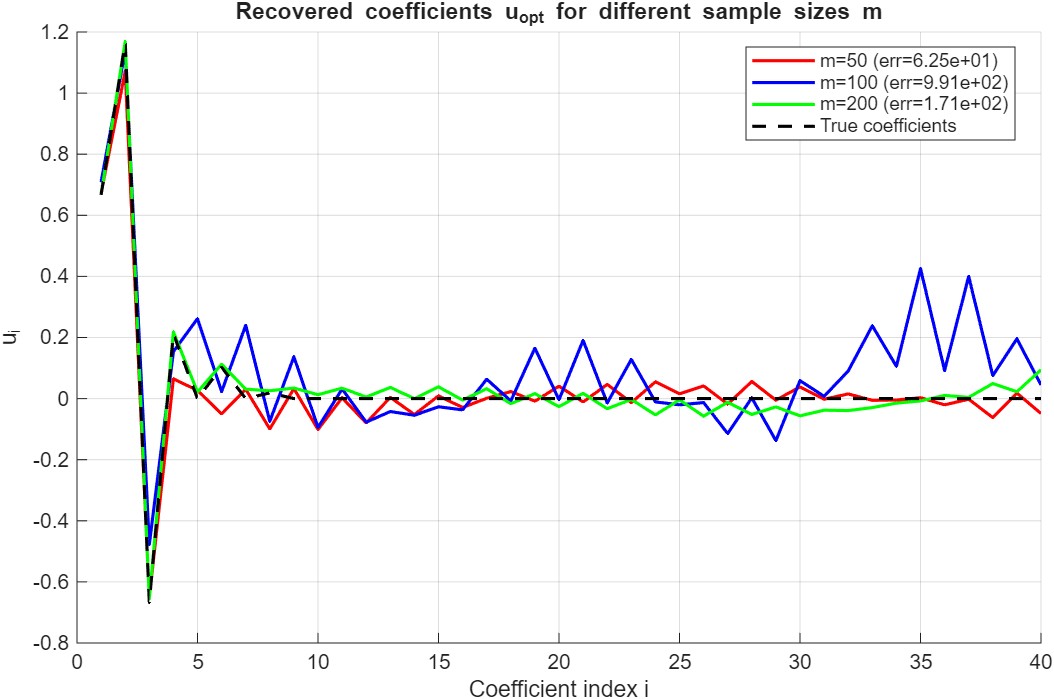}} \subfloat[$\sigma=0.01,m=
	50,100,200$]{\includegraphics[width = 2.75in,
		height=1.7in]{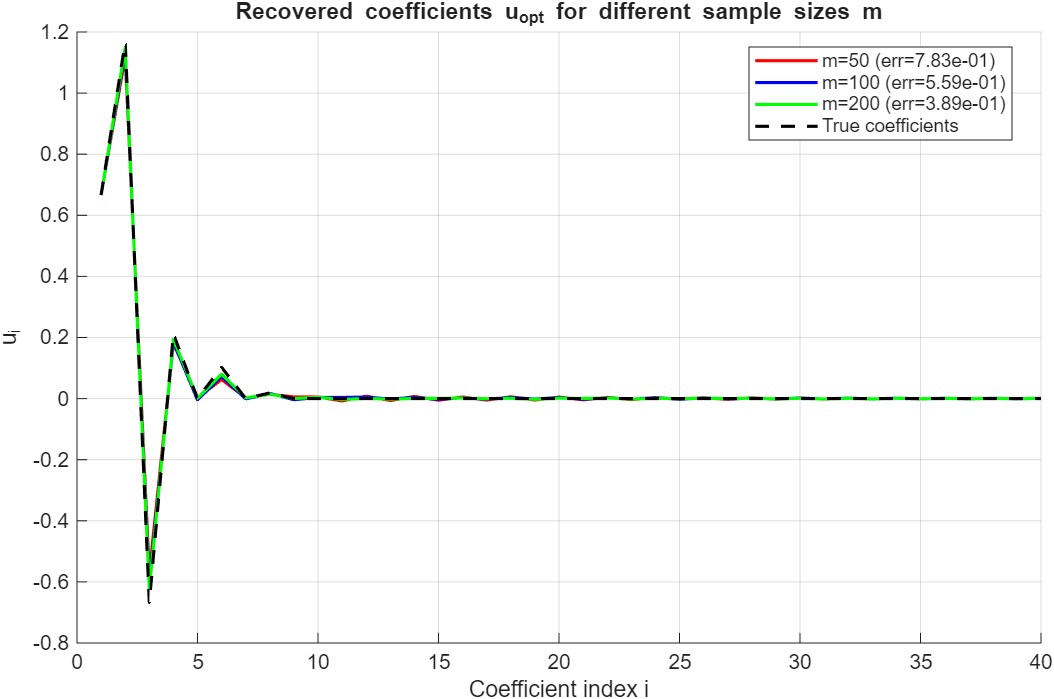}}
	\caption{Computed regularized solutions for $\sigma=0.01$ and different $m$ } \label{fig3}
\end{figure}
%

\begin{table}[H]
	\caption{Computational results for different error levels }
	\begin{center} 
		\begin{tabular}{|c|c|c|c|c|}
			\hline $\sigma$ & $m$  & Regul.Parameter, $\lambda$
			&   \multicolumn{1}{|c|}{Unregul. Error  }
			& \multicolumn{1}{|c|}{Regul. Error}\\
			\hline
			& 50 & $4.6488 \times10^{-4}$  & $1.34\times10^2$  & $1.11\times10^0$ \\
			\cline{2-5}
			$10\%$  & 100 & $2.934\times10^{-4}$  & $2.74\times10^3$ & $9.11\times10^{-1}$ \\
			\cline{2-5}
			& 200& $1.851 \times10^{-4}$  & $1.29\times10^3$ & $7.78\times10^{-1}$\\
			\hline
			& 50 & $4.6488 \times10^{-4}$  & $7.94\times10^1$  & $9.06\times10^{-1}$ \\
			\cline{2-5}
			$5\%$  & 100 & $2.934\times10^{-4}$  & $2.79\times10^3$ & $6.93\times10^{-1}$ \\
			\cline{2-5}
			& 200& $1.851 \times10^{-4}$  & $1.89\times10^3$ & $5.40\times10^{-1}$\\
			\hline
			& 50 & $4.6488 \times10^{-4}$  & $6.25\times10^1$  & $7.83\times10^{-1}$ \\
			\cline{2-5}
			$1\%$  & 100 & $2.934\times10^{-4}$  & $9.91\times10^2$ & $5.59\times10^{-1}$ \\
			\cline{2-5}
			& 200& $1.851 \times10^{-4}$  & $1.71\times10^3$ & $3.89\times10^{-1}$\\
			\hline
		\end{tabular}\label{table1}
	\end{center}
\end{table}

\section{Conclusion}

In this paper, we studied statistical inverse problems in the
Banach space setting and established both convergence and
convergence rate analyses. Furthermore, numerical simulations were
performed to demonstrate the implementability and effectiveness of
the proposed method. It is important to note that, in the
numerical example, the covering number used does not represent a
sharp upper bound; therefore, Theorem 3 was not applied to this
case. 

We make the remark that the rate obtained in this paper is not sharp, and hence does not reflect an optimal rate of convergence. At present, work related to optimality is still being actively investigated by us, and we expect to report sharper, potentially optimal rates in a subsequent study.\\

For a $2$-uniformly convex space such as a Hilbert space
(i.e., when $q=2$), the convergence rate admits an upper bound of
order $m^{-\frac{\beta}{(2+s)[1+\beta]}}$. In \cite{BGM}, the optimal convergence rate in a Reproducing Kernel Hilbert space setting for 
\begin{equation}
	\|L_v^l(u_z-u_{\rho})\|
\end{equation}
is of the order
\begin{equation}
	O\left(\left(\frac{1}{m}\right)^{\frac{b(r+l)}{2br+b+1}}\right)
\end{equation}
Here $l\in[0,1/2]$ and $u_z$ is the regularized solution. $b>1$ is a parameter corresponding to the probability measure and eigen values of the self adjoint operator $L_v$. The parameter $r$ appears in the source condition $f_{\rho}=L_v^r(h)$, for some $h$ in the domain.

For $l=0$ the result is comparable to that of ours and provides a rate of
\begin{equation}
	O\left(\left(\frac{1}{m}\right)^{\frac{br}{2br+b+1}}\right)
\end{equation}
As it can be seen, our result can be made optimal for the special cases when the powers coincide. In particular for the limiting case $s=0$ and for $\beta=2/3$, they coincide, when $b=2$ and $r=1/2$.

The proposed framework can be further generalized by
considering an arbitrary continuous convex function as the
regularization functional. In that case, the error and convergence
rate would be measured using the Bregman distance, allowing the
extension of the analysis to non-reflexive Banach spaces as
domains. In such a study, one could corporate the generalized theory of RKBS as given in \cite{Xu2023SparseBanach}.\\

\noindent\textbf{Acknowledgment}$~$ We sincerely thank the anonymous referee(s) for their careful reading of the manuscript and for their valuable suggestions, which have significantly improved the presentation of the paper.

		\bibliographystyle{plain}
		\footnotesize{\bibliography{reffinal}}

\begin{thebibliography}{10}

\bibitem{BAUER200752}
Frank Bauer, Sergei Pereverzev, and Lorenzo Rosasco.
\newblock On regularization algorithms in learning theory.
\newblock {\em Journal of Complexity}, 23(1):52--72, 2007.

\bibitem{BGM}
Gilles Blanchard and Nicole Mücke.
\newblock Optimal rates for regularization of statistical inverse learning
  problems.
\newblock {\em Foundations of Computational Mathematics}, 18, 04 2016.

\bibitem{Bregman1967}
L.~M. Bregman.
\newblock The relaxation method of finding the common point of convex sets and
  its application to the solution of problems in convex programming.
\newblock {\em USSR Computational Mathematics and Mathematical Physics},
  7(3):200--217, 1967.

\bibitem{Clarke1998}
Francis~H. Clarke, Yuri~S. Ledyaev, Ronald~J. Stern, and Peter~R. Wolenski.
\newblock {\em Nonsmooth Analysis and Control Theory}, volume 178 of {\em
  Graduate Texts in Mathematics}.
\newblock Springer-Verlag, 1998.

\bibitem{CuckerSmale2002}
Felipe Cucker and Steve Smale.
\newblock Best choices for regularization parameters in learning theory: on the
  bias-variance problem.
\newblock {\em Foundations of Computational Mathematics}, 2(4):413--428, 2002.

\bibitem{CuckerZhou2007}
Felipe Cucker and Ding‐Xuan Zhou.
\newblock {\em Learning Theory: An Approximation Theory Viewpoint}.
\newblock Cambridge Monographs on Applied and Computational Mathematics, 24.
  Cambridge University Press, Cambridge, UK, 2007.

\bibitem{dudley1967sizes}
Richard~M. Dudley.
\newblock The sizes of compact subsets of hilbert space and continuity of
  gaussian processes.
\newblock {\em Journal of Functional Analysis}, 1(3):290--330, 1967.

\bibitem{MR463994}
Ivar Ekeland and Roger Temam.
\newblock {\em Convex analysis and variational problems}, volume Vol. 1 of {\em
  Studies in Mathematics and its Applications}.
\newblock North-Holland Publishing Co., Amsterdam-Oxford; American Elsevier
  Publishing Co., Inc., New York, 1976.
\newblock Translated from the French.

\bibitem{engl2000regularization}
H.W. Engl, M.~Hanke, and A.~Neubauer.
\newblock {\em Regularization of Inverse Problems}.
\newblock Mathematics and Its Applications. Springer Netherlands, 2000.

\bibitem{Xu1991Inequalities}
K.~Xu H.\.
\newblock Inequalities in banach spaces with applications.
\newblock {\em Nonlinear Analysis}, 16(12):1127--1138, 1991.

\bibitem{Hanner1956}
Olof Hanner.
\newblock On the uniform convexity of $l^p$ and $\ell^p$.
\newblock {\em Arkiv för Matematik}, 1(5):237--244, 1956.

\bibitem{Hofmann2007}
Bernd Hofmann, Barbara Kaltenbacher, Christiane P{\"o}schl, and Otmar Scherzer.
\newblock A convergence rates result for tikhonov regularization in banach
  spaces with non-smooth operators.
\newblock {\em Inverse Problems}, 23(3):987--1010, 2007.

\bibitem{Liu2017}
Huanxiang Liu, Baohuai Sheng, and Peixin Ye.
\newblock The improved learning rate for regularized regression with rkbss.
\newblock {\em International Journal of Machine Learning and Cybernetics},
  8(4):1235--1245, 2017.

\bibitem{nair2009linear}
M.T. Nair.
\newblock {\em Linear Operator Equations: Approximation and Regularization}.
\newblock World Scientific, 2009.

\bibitem{rajan1}
M.~P. Rajan.
\newblock Convergence analysis of a regularized approximation for solving
  fredholm integral equations of the first kind.
\newblock {\em Journal of mathematical analysis and applications},
  279(2):522--530, 2003.

\bibitem{rajan2}
M.~P. Rajan.
\newblock A modified convergence analysis for solving fredholm integral
  equations of the first kind.
\newblock {\em Integral Equations and Operator Theory}, 49(4):511--516, 2004.

\bibitem{rajan3}
M.~P. Rajan.
\newblock An efficient discretization scheme for solving ill-posed problems.
\newblock {\em Journal of mathematical analysis and applications},
  313(2):654--677, 2006.

\bibitem{rajan4}
M.~P. Rajan.
\newblock A posteriori parameter choice with an efficient discretization scheme
  for solving ill-posed problems.
\newblock {\em Applied mathematics and computation}, 204(2):891--904, 2008.

\bibitem{rajan5}
M.~P. Rajan.
\newblock A parameter choice strategy for the regularized approximation of
  fredholm integral equations of the first kind.
\newblock {\em International Journal of Computer Mathematics},
  87(11):2612--2622, 2010.

\bibitem{rajan6}
M.~P. Rajan.
\newblock An efficient ridge regression algorithm with parameter estimation for
  data analysis in machine learning.
\newblock {\em SN Computer Science}, 3(2):171, 2022.

\bibitem{rajan7}
M.~P. Rajan.
\newblock Convergence analysis of a class of regularization methods with a
  novel discrete scheme for solving inverse problems.
\newblock {\em Journal of Computational Mathematics}, pages 1--16, 2025(
  doi:https://doi.org/10.4208/jcm.2503-m2024-0225).

\bibitem{rockafellar1970convex}
R.~Tyrrell Rockafellar.
\newblock {\em Convex Analysis}.
\newblock Princeton University Press, Princeton, NJ, 1970.

\bibitem{schuster2012regularization}
T.~Schuster, B.~Kaltenbacher, B.~Hofmann, and K.S. Kazimierski.
\newblock {\em Regularization Methods in Banach Spaces}.
\newblock Radon Series on Computational and Applied Mathematics. De Gruyter,
  2012.

\bibitem{Learningrate}
Baohuai Sheng and Peixin Ye.
\newblock The learning rates of regularized regression based on reproducing
  kernel banach spaces.
\newblock {\em Abstract and Applied Analysis}, 2013(1):694181, 2013.

\bibitem{Showalter1997}
Ralph~E. Showalter.
\newblock {\em Monotone Operators in Banach Space and Nonlinear Partial
  Differential Equations}, volume~49 of {\em Mathematical Surveys and
  Monographs}.
\newblock American Mathematical Society, Providence, RI, 1997.

\bibitem{Sprung2019}
Benjamin Sprung.
\newblock Upper and lower bounds for the bregman divergence.
\newblock {\em Journal of Inequalities and Applications}, 2019(1):4, 2019.

\bibitem{Tikhonov:1963}
A.~N. Tikhonov.
\newblock Solution of incorrectly formulated problems and the regularization
  method.
\newblock {\em Soviet Math. Dokl.}, 4(4):1035--1038, 1963.

\bibitem{JMLR:v22:20-751}
Rui Wang and Yuesheng Xu.
\newblock Representer theorems in banach spaces: Minimum norm interpolation,
  regularized learning and semi-discrete inverse problems.
\newblock {\em Journal of Machine Learning Research}, 22(225):1--65, 2021.

\bibitem{MR4749129}
Rui Wang, Yuesheng Xu, and Mingsong Yan.
\newblock Sparse representer theorems for learning in reproducing kernel
  {B}anach spaces.
\newblock {\em J. Mach. Learn. Res.}, 25:Paper No. [93], 45, 2024.

\bibitem{Xu2023SparseBanach}
Yuesheng Xu.
\newblock Sparse machine learning in banach spaces.
\newblock {\em Applied Numerical Mathematics}, 187:185--204, 2023.

\bibitem{XuRoach1991}
Z.~B. Xu and G.~F. Roach.
\newblock Characteristic inequalities of uniformly convex and uniformly smooth
  banach spaces.
\newblock {\em Journal of Mathematical Analysis and Applications},
  157(1):189--210, 1991.

\bibitem{zeidler1985nonlinear}
Eberhard Zeidler.
\newblock {\em Nonlinear Functional Analysis and its Applications, Volume III:
  Variational Methods and Optimization}.
\newblock Springer, New York, 1985.

\bibitem{5179093}
Haizhang Zhang, Yuesheng Xu, and Jun Zhang.
\newblock Reproducing kernel banach spaces for machine learning.
\newblock In {\em 2009 International Joint Conference on Neural Networks},
  pages 3520--3527, 2009.

\end{thebibliography}
		\nocite{*}

	\end{document}